\documentclass[12pt, final]{amsbook}

\usepackage{xy, hyperref, amsmidx}
\xyoption{all}
\hypersetup{pdfauthor={Amritanshu Prasad}, pdftitle={Lecture notes for AIS Pune}}

\DeclareMathSizes{12}{10}{8}{6}

\newtheorem{theorem}{Theorem}[chapter]
\newtheorem*{theorem*}{Theorem}
\newtheorem{lemma}[theorem]{Lemma}
\newtheorem{cor}[theorem]{Corollary}
\newtheorem{prop}[theorem]{Proposition}

\theoremstyle{definition}
\newtheorem{defn}[theorem]{Definition}
\newtheorem*{defn*}{Definition}

\newtheorem{xca}[theorem]{Exercise}

\theoremstyle{remark}
\newtheorem{remark}[theorem]{Remark}

\numberwithin{section}{chapter}
\numberwithin{equation}{chapter}

\newcommand{\Aut}{\mathop{\mathrm{Aut}}\nolimits}
\newcommand{\bsl}{\backslash}
\newcommand{\C}{\mathbf{C}}
\newcommand{\dash}{\nobreakdash-}
\newcommand{\End}{\mathop{\mathrm{End}}\nolimits}
\newcommand{\F}{\mathbf F}
\newcommand{\Fp}{\F_p}
\newcommand{\Fq}{\F_q}
\newcommand{\FQ}{\F_{q^2}}
\newcommand{\FT}{\mathrm{FT}}
\newcommand{\half}{\tfrac{1}{2}}
\newcommand{\Hilbert}{\mathcal H}
\renewcommand{\hat}{\widehat}

\newcommand{\Hom}{\mathop{\mathrm{Hom}}\nolimits}
\newcommand{\inv}{^{-1}}

\newcommand{\Mat}[4]{
  \begin{pmatrix}
    #1 & #2 \\
    #3 & #4
  \end{pmatrix}
}
\newcommand{\mat}[4]{
  \left(
    \begin{smallmatrix}
      #1 & #2 \\
      #3 & #4
    \end{smallmatrix}
  \right)
}
\newcommand{\N}{\mathbf N}
\newcommand{\ol}{\overline}

\newcommand{\tr}{\mathrm{tr}}
\renewcommand{\vec}{\mathbf}
\newcommand{\Z}{\mathbf{Z}}

\makeindex{myindex}

\begin{document}

\frontmatter

\title[Representations of $GL_2(\Fq)$ and $SL_2(\Fq)$]{Representations of \\ $GL_2(\Fq)$ and $SL_2(\Fq)$,\\and some remarks about $GL_n(\Fq)$}

\author{Amritanshu Prasad}
\address{The Institute of Mathematical Sciences\\Chennai.}
\email{amri@imsc.res.in}
\urladdr{\href{http://www.imsc.res.in/\~amri}{http://www.imsc.res.in/\~{}amri}}
\date{July 2007}
\dedicatory{\hspace{1cm}\\Notes from a course taught at\\the Advanced Instructional School on Representation Theory and Related Topics\\held at the Bhaskaracharya Pratishthana and the University of Pune in July 2007}


\date{July 2007}

\maketitle

\setcounter{page}{4}

\tableofcontents

\chapter*{Introduction}

The goal of these notes is to give a self-contained account of the representation theory of $GL_2$ and $SL_2$ over a finite field, and to give some indication of how the theory works for $GL_n$ over a finite field.

Let $\Fq$ denote a finite field with $q$ elements, where $q$ is a prime power.
The irreducible characters of $GL_2(\Fq)$ and $SL_2(\Fq)$ were classified by Herbert~E.~Jordan~\cite{Jordan07} and Issai~Schur~\cite{Schur06} in 1907.
The method used here is not that of Jordan or Schur, but depends on a construction known as the \emph{Weil representation} introduced by Andr\'e Weil in his famous article \cite{MR0165033}.
Weil's method was used to obtain all the irreducible representations of $SL_2(\Fq)$ in \cite{MR0219635} by Shun'ichi Tanaka.
A very readable exposition is also found in Daniel Bump's book \cite[Section~4.1]{MR1431508} in the case of $GL_2(\Fq)$.
These two works have been my main sources.
The use of the Weil representation has the disadvantage that it does not generalise to other groups (such as $GL_n(\Fq)$, $SL_n(\Fq)$, or other finite groups of Lie type, with the exception of $Sp_4(\Fq)$).
On the other hand, the Weil representation is important in number theory as well as representation theory. For example, a version of the Weil representation plays an important role in the construction of supercuspidal representations of reductive groups over non-Archimedean local fields, as was first demonstrated by Takuro Shintani in \cite{MR0233931}.
A systematic use of the Weil representation in this context is made by Paul G{\'e}rardin in \cite{MR0396859}.
These techniques have been used with considerable success to prove the local Langlands conjectures for non-Archimedean local fields, but this is a matter that will not be discussed here.

For $n\times n$ matrices, the representations were classified by James A.~Green in 1955 \cite{MR0072878}.
The general linear groups are special cases of a class of groups known as \emph{reductive groups}, which occur as closed subgroups of general linear groups (in the sense of algebraic geometry).
In 1970, T.~A.~Springer presented a set of conjectures describing the characters of irreducible representations of all reductive groups over finite fields, some of which he attributed to Ian~G.~MacDonald~\cite{MR0263942}.
The essence of these conjectures is that the irreducible representations of reductive groups over finite fields occur in families associated to \emph{maximal tori} in these groups (in this context, a torus is a subgroup that is isomorphic to a product of  multiplicative groups of finite extensions of $\Fq$).
A big breakthrough in this subject came in 1976, when Pierre~Deligne and George~Lusztig~\cite{MR0393266}, were able to construct the characters of almost all the irreducible representations (in an asymptotic sense) of all reductive groups over finite fields, in particular, proving the conjectures of MacDonald.
Much more information about the irreducible representations of reductive groups over finite fields has been obtained in later work, particularly by Lusztig (see e.g., \cite{MR804741}).
The above survey is far from complete and fails to mention many important developments in the subject. 
It is intended only to give the reader a rough sense of where the material to be presented in these lectures lies in the larger context of 20th century mathematics. 

I am grateful to Pooja Singla, who carefully read an earlier version of these notes and pointed out several errors. I have had many interesting discussions with her on the representation theory of $GL_2(\Fq)$, which have helped me when I wrote these notes. I am grateful to M.~K.~Vemuri, from whom I have learned a large part of what I know about Heisenberg groups and Weil representations.

\mainmatter
\chapter{General results from representation theory}
\label{cha:general-results}
\section{Basic definitions}
\label{sec:defintions}

Let $G$ be a finite group.
A \emph{representation}\index{myindex}{representation} of $G$ on a vector space $V$ is a pair $(\pi,V)$ where $V$ is a complex vector space and $\pi$ is a homomorphism $G\to GL(V)$.
Often, we will denote $(\pi,V)$ simply by $\pi$, specially when the vector space $V$ is specified implicitly.
The dimension of $V$ is called the \emph{degree}\index{myindex}{representation!degree of a} of the representation $(\pi,V)$.
In these notes all representations will be assumed to be of finite degree.
If $(\pi,V)$ and $(\tau,U)$ are two representations of $G$, then a linear map $\phi:U\to V$ is called a homomorphism\index{myindex}{representation!homomorphism of} of $G$-modules, or an \emph{intertwiner}\index{myindex}{intertwiner} if
\begin{equation*}
  \phi(\tau(g)u)=\pi(g)\phi(u) \mbox{ for all } u\in U.
\end{equation*}
The space of all homomorphisms $(\tau,U)\to (\pi,V)$ will be denoted by $\Hom_G(\tau,\pi)$.
When $\phi$ is invertible, it is an isomorphism, and we say that $\tau$ is \emph{isomorphic} to $\phi$\index{myindex}{isomorphism of representations}.
The representations $\pi$ and $\tau$ are said to be \emph{disjoint}\index{myindex}{disjoint representations} if $\Hom_G(\tau,\pi)=0$.
\section{The Pontryagin dual of a finite abelian group}
\label{sec:pontryagin_dual}
\index{myindex}{Pontryagin-dual}
Let $G$ be an abelian group.
The binary operation on the group will be written additively.
A \emph{character} of $G$ is a homomorphism $\chi:G\to\C^*$.
In other words, $\chi(x+x')=\chi(x)\chi(x')$ for all $x,x'\in G$.
A character $\chi$ is called \emph{unitary} if $|\chi(x)|=1$ for all $x\in G$.
\begin{xca}
  Show that every character of a finite abelian group is unitary. 
\end{xca}
If $G$ is a finite abelian group, its \emph{Pontryagin dual} is the set $\hat{G}$ of its characters.
Under point-wise multiplication of characters, $\hat{G}$ forms a group.
Once again, the binary operation is written additively, so that given characters $\chi$ and $\chi'$ of $G$, $(\chi+\chi')(x)=\chi(x)\chi'(x)$ for all $x\in G$.
This is a special case of a general construction for \emph{locally compact abelian groups}.
\begin{prop}
  \label{prop:Pontryagin-duality}
  For any finite abelian group $G$, $G\cong \hat{G}$.
\end{prop}
\begin{proof}
  The proof is a sequence of exercises:
  \begin{xca}
    Show that the Proposition is true for a finite cyclic group $\Z/n\Z$.
  \end{xca}
  \begin{xca}
    If $G_1$ and $G_2$ are abelian groups, show that
    \begin{equation*}
      \widehat{G_1\times G_2}\cong\hat{G}_1\times \hat{G}_2.
    \end{equation*}
  \end{xca}
  \begin{xca}
    Show that every finite abelian group is isomorphic to a product of finite cyclic groups.
  \end{xca}
\end{proof}
It follows from the above proposition that $\hat{\hat{G}}\cong G$.
However, in this case, there is a \emph{canonical} isomorphism $G\to \hat{\hat{G}}$ given by $g\mapsto \check{g}$ where $\check{g}$ is defined by
\begin{equation*}
  \check{g}(\chi)=\chi(g)\mbox{ for each }\chi \in \hat{G}.
\end{equation*}
\section{Induced Representations}
\label{sec:induced}
Let $H$ be a subgroup of $G$.
Given a representation $(\pi,V)$ of $H$, the representation of $G$ \emph{induced}\index{myindex}{induction}\index{myindex}{representation!induced} from $\pi$ is the representation $(\pi^G,V^G)$ where
\begin{equation*}
  V^G=\{f:G\to V|f(hg)=\pi(h)f(g) \mbox{ for all } h\in H, g\in G\}.
\end{equation*}
The action of $G$ on such functions is by \emph{right translation}
\begin{equation*}
  (\pi^G(g)f)(x)=f(xg).
\end{equation*}

Now suppose that $(\tau,U)$ is a representation of $G$ and $(\pi,V)$ is a representation of $H$.
Because $H\subset G$, we can regard $U$ as a representation of $H$ by restricting the homomorphism $G\to GL(U)$ to $H$.
Denote this representation by $\tau_H$.
  Given $\phi\in \Hom_G(\tau,\pi^G)$, define $\tilde \phi:U\to V$ by
  \begin{equation*}
    \tilde{\phi}(u)=\phi(u)(1)\mbox{ for each } u\in U.
  \end{equation*}
  \begin{xca}
    Show that $\tilde\phi\in \Hom_H(\tau_H,\pi)$.
  \end{xca}
\begin{theorem*}[\hypertarget{Frobenius-reciprocity}{Frobenius reciprocity}]
  The map $\phi\mapsto \tilde{\phi}$ induces an isomorphism
  \begin{equation*}
    \Hom_G(\tau,\pi^G)\tilde{\to}\Hom_H(\tau_H,\pi).
  \end{equation*}
\end{theorem*}
\begin{proof}
  For $\psi\in \Hom_H(\tau_H,\pi)$ define $\tilde{\psi}:U\to V^G$ by
  \begin{equation*}
    \tilde{\psi}(u)(x)=\psi(\tau(x)u) \mbox{ for each } u\in U \mbox{ and } x\in G.
  \end{equation*}
  \begin{xca}
    For all $h\in H$, $\tilde{\psi}(u)(hx)=\pi(h)\tilde{\psi}(u)(x)$.
    Therefore, $\tilde{\psi}(u)\in V^G$.
  \end{xca}
  \begin{xca}
    Show that $\tilde \psi\in \Hom_G(\tau,\pi^G)$.
  \end{xca}
  \begin{xca}
    For all $\phi\in \Hom_G(\tau,\pi^G)$, $\tilde{\tilde{\phi}}=\phi$, and for all $\psi \in \Hom_H(\tau_H,\pi)$, $\tilde{\tilde{\psi}}=\psi$.
  \end{xca}
  Therefore the maps $\phi\mapsto\tilde{\phi}$ and $\psi\mapsto \tilde{\psi}$ are mutual inverses.
\end{proof}

\section{Description of intertwiners}
\label{sec:Mackey}
In this section we describe the homomorphisms between two induced representations.
Let $G$ be a finite group.
Let $H_1$ and $H_2$ be subgroups.
Let $(\pi_1,V_1)$ and $(\pi_2,V_2)$ be representations of $H_1$ and $H_2$ respectively.
For $f:G\to V_1$, and $\Delta:G\to \Hom_\C(V_1,V_2)$, define a convolution $\Delta*f:G\to V_2$ by
\begin{equation*}
  (\Delta*f)(x)=\frac{1}{|G|}\sum_{g\in G}\Delta(xg\inv)f(g).
\end{equation*}
Let $D$ be the set of all functions $\Delta:G\to \Hom_\C(V_1,V_2)$ satisfying
\begin{equation*}
  \Delta(h_2gh_1)=\pi_2(h_2)\circ \Delta(g)\circ \pi_1(h_1)
\end{equation*}
for all $h_1\in H_1$, $h_2\in H_2$ and $g\in G$.
\begin{xca}
  Show that if $\Delta\in D$ and $f_1\in V_1^G$ then $\Delta*f_1\in V_2^G$.
\end{xca}
\begin{xca}
  Show that the map $L_\Delta:V_1^G\to V_2^G$ defined by $f_1\mapsto \Delta*f_1$ is a homomorphism of $G$-modules.
\end{xca}
\begin{theorem}[Mackey]
  \label{theorem:Mackey}
  The map $\Delta\mapsto L_\Delta$ is an isomorphism from $D\to \Hom_G(V_1^G,V_2^G)$.
\end{theorem}
\begin{proof}
  We construct an inverse mapping $\Hom_G(V_1^G,V_2^G)\to D$.
  For this, let us define a collection $f_{g,v}$ of elements in $V_1^G$ indexed by $g\in G$ and $v\in V_1$:
  \begin{equation*}
    f_{g,v}(x)=
    \begin{cases}
      \pi_1(h)v & \mbox{ if } x=hg, h\in H_1\\
      0 & \mbox{ if } x\notin H_1g.
    \end{cases}
  \end{equation*}
  \begin{xca}
    Show that for every $v\in V_1$, we have
    \begin{equation*}
      \Delta(g)(v)=[G:H_1]L_\Delta(f_{g\inv, v})(1).
    \end{equation*}
  \end{xca}
  The above equation can be turned around to define, for each $L:\Hom_G(V_1^G,V_2^G)$ a function $\Delta\in D$.
  \begin{xca}
    Show that if $L\in \Hom_G(V_1,V_2)$, then the function $\Delta:G\to \Hom_C(V_1,V_2)$ defined by
    \begin{equation*}
      \Delta_L(g)(v)=[G:H_1]L(f_{g\inv, v})(1)
    \end{equation*}
    is in $D$.
  \end{xca}
  \begin{xca}
    Check that the maps $\Delta\mapsto \Delta_L$ and $L\mapsto L_\Delta$ are inverses of each other.
  \end{xca}
\end{proof}
\section{A criterion for irreducibility}
Let $G$ be a finite group, $H$ a subgroup and $(\pi,V)$ a representation of $H$.
The space $V^G$ can be decomposed into a direct sum
\begin{equation*}
  V^G = \bigoplus_{Hx\inv H\in H\bsl G/H} V_{Hx\inv H},
\end{equation*}
where $V_{Hx\inv H}$ consists of functions $G\to V$ supported on $Hx\inv H$:
\begin{equation*}
  V_{Hx\inv H}=\{f:Hx\inv H\to V\;|\:f(hx\inv h')=\pi(h)f(x\inv h') \text{ for all } h,h'\in H\}.
\end{equation*}
$V_{Hx\inv H}$ is stable under the action of $\pi$.
Let $\pi_{Hx\inv H}$ denote the resulting representation of $H$ on $V_{Hx\inv H}$ and let ${}^x\pi_{H\cap x\inv Hx}$ denote the representation of $H\cap xHx\inv$ on $V$ given by ${}^x\pi(h)=\pi(x\inv hx)$.
\begin{xca}
  Show that $f\mapsto (h\mapsto f(x\inv h))$ defines an isomorphism of representations
  \begin{equation*}
    \pi_{Hx\inv H}\cong ({}^x \pi_{H\cap xHx\inv})^H.
  \end{equation*}
\end{xca}
We have proved
\begin{prop}
  \label{prop:indes}
  Let $G$ be a finite group and $H$ any subgroup.
  For every representation $\pi$ of $H$, there is a canonical isomorphism of representations of $H$
  \begin{equation*}
    (\pi^G)_H=\bigoplus_{Hx\inv H\in H\bsl G/H} ({}^x \pi_{H\cap xHx\inv})^H.
  \end{equation*}
\end{prop}
By \hyperlink{Frobenius-reciprocity}{Frobenius reciprocity},
\begin{eqnarray*}
  \End_G(\pi^G)&=&\Hom_H((\pi^G)_H,\pi)\\
  & = & \bigoplus_{Hx\inv H\in H\bsl G/H} \Hom_H(({}^x\pi_{H\cap xHx\inv})^H,\pi).
\end{eqnarray*}
Recall that $\pi^G$ is irreducible if and only if $\End_G(\pi^G)$ is one dimensional.
As a result, we obtain \emph{Mackey's irreducibility criterion}\index{myindex}{Mackey's irreducibility criterion}:
\begin{theorem*}
  [Mackey's irreducibility criterion]
  Let $G$ be a finite group and $H$ a subgroup.
  Let $\pi$ be an irreducible representation of $H$.
  Then $\pi^G$ is irreducible if and only if, for any $x\notin H$, the representations $\pi$ and $({}^x\pi_{H\cap xHx\inv})^H$ are disjoint.
\end{theorem*}
\begin{cor}
  \label{cor:Mackeycor}
  Suppose that $G$ is a finite group and $H$ a normal subgroup.
  Then for any irreducible representation $\pi$ of $H$, $\pi^G$ is irreducible if and only if for every $x\notin H$, ${}^x\pi$ is not isomorphic to $\pi$.
\end{cor}
\section{The little groups method}
The little groups method was first used by Wigner \cite{MR1503456}, and generalized by Mackey \cite{MR0098328} to construct representations of a group from those of a normal subgroup.
We will restrict ourselves to the case where $G$ is a finite group and $N$ is a normal subgroup of $G$ which is abelian.
Let $\hat N$ denote the Pontryagin dual of $N$ (Section~\ref{sec:pontryagin_dual}).
Define an action of $G$ on $\hat N$ by
\begin{equation*}
  {}^g\chi(n)=\chi(g\inv ng) \text{ for each } g\in G,\; \chi\in \hat N.
\end{equation*}
Let $\rho$ be an irreducible representation of $G$ on the vector space $V_\rho$.
For each $\chi \in \hat N$, write
\begin{equation*}
  V_\chi = \{ \vec x\in V\;|\: \rho(n)\vec x = \chi(n) \vec x\}.
\end{equation*}
Then
\begin{equation*}
  V_\rho = \bigoplus_{\chi \in \hat N} V_\chi.
\end{equation*}
Define
\begin{equation*}
  \hat N (\rho ) = \{ \chi\in \hat N\;|\: V_\chi \neq 0\}.
\end{equation*}
\begin{prop}
  [Clifford's theorem]
  \label{prop:Clifford}
  $\hat N(\rho)$ consists of a single $G$\dash orbit of $\hat N$.
\end{prop}
\begin{proof}
  Suppose $\vec x \in V_\chi$, and $g\in G$.
  Then
  \begin{eqnarray*}
    \rho(n)(\rho(g)\vec x) & = & \rho(g)\rho(g\inv ng)\vec x\\
    & = & {}^g\chi \rho(g)\vec x.
  \end{eqnarray*}
  Therefore, 
  \begin{equation}
    \label{eq:system-of-imprimitivity}
    \rho(g)V_\chi= V_{{}^g\chi}.
  \end{equation}
  It follows that $\oplus_{g\in G} V_{{}^g\chi}$ is invariant under $\rho$.
  From the irreducibility of $\rho$ one concludes that if $V_\chi\neq 0$, then
  $\oplus_{g\in G} V_{{}^g\chi}=V_\rho$.
\end{proof}
For $\chi\in \hat N(\rho)$, let 
\begin{equation*}
  G_\chi=\{g\in G\;|\:{}^g\chi = \chi\}.
\end{equation*}
It follows from (\ref{eq:system-of-imprimitivity}) that for every $g\in G_\chi$, $\rho(g)$ preserves $V_\chi$.
Therefore, $\rho$ gives rise to a representation $\rho_\chi$ of $G_\chi$ on $V_\chi$.
\begin{prop}
  [Mackey's imprimitivity theorem]
  \index{myindex}{Mackey's imprimitivity theorem}
  \label{prop:Mackey-imprimitivity}
  \begin{equation*}
    \rho \cong \rho_\chi^G.
  \end{equation*}
\end{prop}
\begin{proof}
  \begin{equation*}
    V_\rho =\bigoplus_{gG_\chi\in G/G_\chi} V_{{}^g\chi}.
  \end{equation*}
  Therefore, for each $\vec x\in V_\rho$, there is a unique decomposition
  \begin{equation*}
    \vec x = \sum_{G_\chi g\in G_\chi\bsl G} \vec x_{gG_\chi}.
  \end{equation*}
  By (\ref{eq:system-of-imprimitivity}), $\rho(g\inv)\vec x_{gG_\chi}\in V_\chi$.
  The representation space of $\rho_\chi^G$ is
  \begin{equation*}
    V_\chi^G=\{f:G\to \C\;|\:f(g'g)=\chi(g')f(g)\text{ for all } g'\in G_\chi \; g\in G\}.
  \end{equation*}
  Define $\phi(\vec x)(g)=\rho(g)\vec x_{g\inv G_\chi}$ for each $g\in G$.
  \begin{xca}
    Show that $\phi:V_\rho\to \rho_\chi^G$ is a well defined isomorphism of representations of $G$.
  \end{xca}
\end{proof}

\chapter{Representations constructed by\\ { }parabolic induction}
\label{chap:induced_representations}
\section{Conjugacy classes in $GL_2(\Fq)$}
\label{sec:conjugacy_classes}

Given a matrix $\mat{a}{b}{c}{d}$ in $GL_2(\Fq)$ consider its characteristic polynomial
\begin{equation*}
  \lambda^2-(a+d)\lambda+(ad-bc).
\end{equation*}
\begin{xca}
  If the roots $(\lambda_1,\lambda_2)$ are distinct in $\Fq$ then the matrix is conjugate to $\mat{\lambda_1}{0}{0}{\lambda_2}$ and to $\mat{\lambda_2}{0}{0}{\lambda_1}$.  
\end{xca}
\begin{xca}
  If $\lambda_1=\lambda_2$ then, either the matrix is $\mat{\lambda_1}{0}{0}{\lambda_1}$ or it is conjugate to $\mat{\lambda_1}{1}{0}{\lambda_1}$ in $GL_2(\Fq)$.
\end{xca}
\begin{xca}
  \label{ex:anisotropic}
  If $\lambda^2-(a+d)\lambda+(ad-bc)$ is irreducible in $\Fq[t]$, then the matrix is similar to $\mat 0{-(ad-bc)}1{a+d}$.
\end{xca}
To summarise,
the conjugacy classes in $GL_2(\Fq)$ are as follows:
\begin{enumerate}
\item $(q-1)$ classes represented by $\mat{\lambda}{0}{0}{\lambda}$, with $\lambda\in \Fq^*$.
\item $(q-1)$ classes represented by  $\mat{\lambda}{1}{0}{\lambda}$, with $\lambda\in \Fq^*$.
\item $\half(q-1)(q-2)$ classes represented by $\mat{\lambda_1}{0}{0}{\lambda_2}$ with $\lambda_1\neq \lambda_2$.
\item $\half(q^2-q)$ classes represented by $\mat 0{-a_0}1{a_1}$, with $\lambda^2-a_1\lambda+a_0$ an irreducible polynomial in $\Fq[t]$.
\end{enumerate}
In all, there are
\begin{equation}
  \label{eq:conjugacy_classes}
  (q-1)+(q-1)+\frac{q^2-q}{2}+\frac{(q-1)(q-2)}{2}
\end{equation}
conjugacy classes.
Detailed information about the conjugacy classes is collected in Table~\ref{table:GL2-conjugacy-classes}.
\newcommand{\cbox}[2]{\parbox{#1}{\begin{center}#2\end{center}}}
\begin{table}[htb]
    \caption{Conjugacy classes of $GL_2(\Fq)$}
    \label{table:GL2-conjugacy-classes}
    \begin{tabular}{|p{2.2cm}|c|c|c|c|}
      \hline
      Name & element & centraliser & \cbox{1.5cm}{no. of classes} & \cbox{1.5cm}{size of class} \\
      \hline \hline
      Central & $\mat a00a$, $a\in \Fq^*$ & $q(q-1)^2(q+1)$ & $q-1$ & $1$\\
      \hline
      Non-semisimple & $\mat a10a$, $a\in \Fq^*$ & $q(q-1)$ & $q-1$ & $(q-1)(q+1)$\\
      \hline
      Split regular semisimple &
      \cbox{2.4cm}{
        \begin{center}
          $\mat a00b$\\ $a\neq b\in \Fq^*$
        \end{center}
      }
      & $(q-1)^2$ & $\frac{(q-1)(q-2)}{2}$ & $q(q+1)$ \\
      \hline
      Anisotropic regular semisimple & \cbox{2.2cm}{$C_p$, $p(t)\in \Fq[t]$ quadratic, irreducible} & $(q-1)(q+1)$ & $\frac{q^2-q}{2}$ & $q(q-1)$ \\
      \hline
    \end{tabular}
\end{table}

\section{Subgroup of upper-triangular matrices}
\label{sec:upper-triangular}

Let $B$ be the subgroup of $GL_2(\Fq)$ consisting of invertible upper triangular matrices.
Let $N$ be the subgroup of upper triangular matrices with $1$'s along the diagonal.
Let $T$ be the subgroup of invertible diagonal matrices.
\begin{xca}
  \label{xca:subgr-upper-triang}
  Show that
  \begin{enumerate}
  \item Every element $b\in B$ can be written in a unique way as $b=tn$, with $t\in T$ and $n\in N$.
  \item $N$ is a normal subgroup of $B$.
  \item $B/N\cong T$.
  \end{enumerate}
\end{xca}

Let $w=\mat{0}{1}{-1}{0}$.
\begin{prop}[Bruhat decomposition]
  \label{prop:Bruhat_decomposition}
  \begin{equation*}
    GL_2(\Fq)=B\cup BwB, \mbox{ a disjoint union.}
  \end{equation*}
\end{prop}
Note that $B$ is really a double coset $B1B$.
So Proposition~\ref{prop:Bruhat_decomposition} really tells us that the double coset space $B\bsl GL_2(\Fq)/B$ has two elements and that $\{1,w\}$ is a complete set of representatives for these double cosets.
\begin{proof}
A matrix $\mat{a}{b}{c}{d}$ lies in $B$ if and only if $c=0$.
If $c\neq 0$, then
\begin{equation*}
  \mat{a}{b}{c}{d}=\mat{1}{a/c}{0}{1}w\mat{-c}{-d}{0}{b-ad/c}\in BwB.
\end{equation*}
\end{proof}
\section{Parabolically induced representations for $GL_2(\Fq)$}
\label{sec:induced_representations}
Given characters $\chi_1$ and $\chi_2$ of $\Fq^*$, we get a character $\chi$ of $T$ by
\begin{equation*}
  \chi \mat{y_1}{0}{0}{y_2}=\chi_1(y_1)\chi_2(y_2).
\end{equation*}
We extend $\chi$ to a character of $B$ by letting $N$ lie in the kernel.
Thus
\begin{equation}
  \label{eq:character}
  \chi\mat{y_1}{x}{0}{y_2}=\chi_1(y_1)\chi_2(y_2).
\end{equation}
Let $I(\chi_1,\chi_2)$ be the representation of $GL_2(\Fq)$ induced from this character of $B$.
\begin{prop}
  \label{prop:intertwiners}
  Let $\chi_1$, $\chi_2$, $\mu_1$ and $\mu_2$ be characters of $\Fq^*$. 
  Then
  \begin{equation*}
    \dim \Hom_{GL_2(\Fq)}(I(\chi_1,\chi_2),I(\mu_1,\mu_2))=e_1+e_w,
  \end{equation*}
  where,
  \begin{equation*}
    e_1=
    \begin{cases}
      1 & \mbox{ if } \chi_1 = \mu_1 \mbox{ and } \chi_2=\mu_2,\\
      0 & \mbox{ otherwise,}
    \end{cases}
  \end{equation*}
and
\begin{equation*}
    e_w=
    \begin{cases}
      1 & \mbox{ if } \chi_1 = \mu_2 \mbox{ and } \chi_2=\mu_1,\\
      0 & \mbox{ otherwise.}
    \end{cases}
\end{equation*}
\end{prop}
\begin{proof}
  Let $\chi$ and $\mu$ be the characters of $B$ obtained from the pairs $\chi_1$, $\chi_2$ and $\mu_1$, $\mu_2$ respectively as in \eqref{eq:character}.
  We regard $\chi$ and $\mu$ as one-dimensional representations of $B$ acting on the space $\C$.
  We may identify $\Hom_\C(\C,\C)$ with $\C$ as well.
  Then, using Mackey's description of intertwiners (Theorem~\ref{theorem:Mackey}), we see that we must compute the dimension of the space of functions $\Delta:GL_2(\Fq)\to \C$ such that
  \begin{equation}
    \label{eq:Mackey}
    \Delta(b_2gb_1)=\mu(b_2)\Delta(g)\chi(b_1), \mbox{ } b_i\in B.
  \end{equation}
  It follows from the Bruhat decomposition that $\Delta$ is completely determined by its values at $1$ and $w$.

  Taking $g=1$ in \eqref{eq:Mackey}, we see that for any $t\in T$,
  \begin{equation*}
    \mu(t)\Delta(1)=\Delta(t)=\Delta(1)\chi(t)
  \end{equation*}
  Therefore, if $\mu\neq \chi$ then $\Delta(1)=0$.
  On the other hand, if $\mu=\chi$, let $\Delta_1$ be the function such that
  \begin{equation*}
    \Delta_1(b)=\chi(b) \mbox{ for all } b\in B,
  \end{equation*}
  and whose restriction to $BsB$ is zero.
  If $e_1=0$, we take $\Delta_1\equiv 0$.

  Taking $g=w$ in \eqref{eq:Mackey}, we see that for any $t\in T$,
  \begin{equation*}
    \mu(t)\Delta(w)=\Delta(tw)=\Delta(w(w\inv t w))=\Delta(w)\chi(w\inv t w).
  \end{equation*}
  \begin{xca}
    $w\inv \mat{t_1}{0}{0}{t_2}w=\mat{t_2}{0}{0}{t_1}$.
  \end{xca}
  Therefore, if $\mu_1\neq \chi_2$ or $\mu_2\neq \chi_1$ then $\Delta(w)=0$.
  On the other hand, if $\mu_1=\chi_2$ and $\mu_2=\chi_1$, let $\Delta_w$ be the function such that
  \begin{equation*}
    \Delta_w(b_2wb_1)=\chi(b_1)\mu(b_2) \mbox{ for all } b_1,b_2\in B,
  \end{equation*}
  and whose restriction to $B$ is $0$.
  If $e_w=0$, we take $\Delta_w\equiv 0$.

  An arbitrary functions satisfying \eqref{eq:Mackey} can be expressed as a linear combination of $\Delta_1$ and $\Delta_w$, so we see that the dimension of the space of such functions must be $e_1+e_w$.
\end{proof}

\begin{theorem}
  \label{theorem:principal_series}
  Let $\chi_1$, $\chi_2$, $\mu_1$ and $\mu_2$ be characters of $\Fq^*$.
  Then $I(\chi_1,\chi_2)$ is an irreducible representation of degree $q+1$ of $GL_2(\Fq)$ unless
  $\chi_1=\chi_2$, in which case it is a direct sum of two irreducible representations having degrees $1$ and $q$.
  We have
  \begin{equation*}
    I(\chi_1,\chi_2)\cong I(\mu_1,\mu_2)
  \end{equation*}
  if and only if either
  \begin{equation}
    \label{eq:1}
    \chi_1=\mu_1 \mbox{ and } \chi_2=\mu_2
  \end{equation}
  or else
  \begin{equation}
    \label{eq:s}
    \chi_1=\mu_2 \mbox{ and } \chi_2=\mu_1.
  \end{equation}
\end{theorem}
\begin{proof}
  Apply Proposition~\ref{prop:intertwiners} with $\chi_1=\mu_1$ and $\chi_2=\mu_2$.
  We see that
  \begin{equation*}
    \dim \End_{GL_n(\Fq)}(I(\chi_1,\chi_2))=
    \begin{cases}
      1 & \mbox{ if } \chi_1\neq \chi_2,\\
      2 & \mbox{ if } \chi_1=\chi_2.
    \end{cases}
  \end{equation*}
  Recall that if $(\pi,V)$ is a representation of a finite group $G$ and $V$ is a direct sum of distinct irreducible representations $\pi_1,\cdots,\pi_h$ with multiplicities $m_1,\cdots,m_h$ and with degrees $d_1,\cdots,d_h$ respectively, then the dimension of $\End_G(V)$ is $\sum m_id_i^2$.
  Hence $I(\chi_1,\chi_2)$ is irreducible if $\chi_1\neq \chi_2$, otherwise it is a direct sum of two irreducible representations because $2=1^2+1^2$ is the only way of writing $2$ as a sum of non-zero multiples of more than one non-zero squares.

  Because the index of $B$ in $GL_2(\Fq)$ is $q+1$, the dimension of $I(\chi_1,\chi_2)$ is always $q+1$.
  If $\chi_1=\chi_2$, the representation of $GL_2(\Fq)$ generated by the function $f(g)=\chi_1(\det(g))$ clearly satisfies $f(bg)=\chi(b)f(g)$ for all $b\in B$ and $g\in G$.
  Therefore $f\in I(\chi_1,\chi_2)$.
  Moreover, $(g\cdot f)(x)=\chi_1(\det(g))f$.
  Therefore the one-dimensional subspace spanned by $f$ is invariant under the action of $G$, hence forms a one-dimensional representation of $G$.
  The other component is therefore $q$-dimensional.
  
  If $\chi_1\neq \chi_2$ then, $I(\chi_1,\chi_2)$ is irreducible.
  By Proposition~\ref{prop:intertwiners} there exists a non-zero element in $\Hom(I(\chi_1,\chi_2), I(\mu_1,\mu_2))$ if and only if $\chi_1=\mu_1$ and $\chi_2=\mu_2$ or $\chi_1=\mu_2$ and $\chi_2=\mu_1$.
  By irreducibility, these homomorphisms must be isomorphisms.
  This proves the second part of the theorem.
\end{proof}
\begin{xca}
  Find the isomorphism $I(\chi_1,\chi_2)\to I(\chi_2,\chi_1)$ explicitly, when $\chi_1\neq \chi_2$.
\end{xca}
To summarise, in this section, we have constructed irreducible representations of $GL_2(\Fq)$ corresponding to characters $\chi=(\chi_1,\chi_2)$ of $T$:
\begin{enumerate}
\item When $\chi_1\neq \chi_2$, there is a unique irreducible representation of $GL_2(\Fq)$ of degree $q+1$ corresponding to $\chi$; 
  the irreducible representation corresponding to $(\chi_1,\chi_2)$ is isomorphic to the one corresponding to $(\chi_2,\chi_1)$. 
  We have $\frac{1}{2}(q-1)(q-2)$ irreducible representations of degree $q+1$.
\item When $\chi_1=\chi_2$, there are two irreducible representations of $GL_2(\Fq)$ corresponding to $\chi$, one of degree $1$ and the other of degree $q$.
  All these representations are pairwise non-isomorphic.
  Therefore we have $q-1$ representations of degree $1$ and $q-1$ representations of degree $q$.
\end{enumerate}

Recall from Schur theory, that the number of irreducible representations is the same as the number of conjugacy classes in a group.
We have constructed
\begin{equation*}
  (q-1)+(q-1)+\frac{(q-1)(q-2)}{2}
\end{equation*}
irreducible representations so far.
Comparing with \eqref{eq:conjugacy_classes}, we see that there remain $\half(q^2-q)$ representations left to construct.

Recall that for a group of order $n$ whose irreducible representations are $\pi_1,\ldots, \pi_r$ of degrees $d_1,\ldots,d_r$ respectively,
\begin{equation*}
  n=d_1^2+\cdots + d_r^2.
\end{equation*}
\begin{xca}
  Show that the order of $GL_2(\Fq)$ is $(q^2-1)(q^2-q)$.
\end{xca}
The sum of squares of degrees of the representations that we have constructed so far is
\begin{equation*}
  \half(q-1)(q-2)(q+1)^2+(q-1)(q^2+1).
\end{equation*}
The difference between the above numbers is
\begin{equation*}
  \half(q^2-q)(q-1)^2.
\end{equation*}

We will see in Section~\ref{sec:degrees_cuspidal_representation} that there are $\half(q^2-q)$ irreducible representations of degree $q-1$ remaining.
These will be constructed in Section~\ref{sec:constr-cusp-repr}.
\section{Conjugacy classes in $SL_2(\Fq)$}
Let $\Aut(SL_2(\Fq))$ denote the group of all automorphisms of $SL_2(\Fq)$.
$GL_2(\Fq)$ acts on $SL_2(\Fq)$ by conjugation.
This gives rise to a homomorphism $GL_2(\Fq)\to \Aut(SL_2(\Fq))$.
The kernel of this automorphism consists of scalar matrices in $GL_2(\Fq)$, and is therefore isomorphic to $\Fq^*$.
The image is therefore isomorphic to the group $PGL_2(\Fq)$, which is the quotient of $GL_2(\Fq)$ by the subgroup of invertible scalar matrices.
The orbits of $PGL_2(\Fq)$ on $SL_2(\Fq)$ are precisely the conjugacy classes of $GL_2(\Fq)$ which are contained in $SL_2(\Fq)$ (note that $SL_2(\Fq)$ is a union of conjugacy classes of $GL_2(\Fq)$).

On the other hand, the image of $SL_2(\Fq)$ in $\Aut(SL_2(\Fq))$ is $PSL_2(\Fq)$, the quotient of $SL_2(\Fq)$ by the subgroup $\{\pm 1\}$.
The conjugacy classes of $SL_2(\Fq)$ are precisely the $PSL_2(\Fq)$ orbits.

Now, $PSL_2(\Fq)$ is a subgroup of $PGL_2(\Fq)$ (when both groups are viewed as subgroups of $\Aut(SL_2(\Fq))$) of index two.
Therefore, each conjugacy class of $GL_2(\Fq)$ whose elements lie in $SL_2(\Fq)$ is either a single conjugacy class in $SL_2(\Fq)$ or a union of two conjugacy classes in $SL_2(\Fq)$.
If $\epsilon$ is an element of $\Fq^*$ which is not a square (since $q$ is assumed to be odd, there are $\tfrac{q-1}2$ such elements), then the image of $\mat \epsilon 001$ in $\Aut(SL_2(\Fq))$ does not lie in the image of $SL_2(\Fq)$.

Let $\sigma\in SL_2(\Fq)$.
Whether or not the conjugacy class of $\sigma$ in $GL_2(\Fq)$ splits or not can be determined by counting.
The basic principle here is that the number of elements in an orbit for a group action is the index of the stabiliser of a point in the orbit.

With the above observations in mind, it is not difficult to prove that
\begin{theorem}
  \label{theorem:orbit-splitting}
  Let $\sigma\in SL_2(\Fq)$.
  Let $Z$ denote the centraliser of $\sigma$ in $GL_2(\Fq)$.
  Then $[Z:Z\cap SL_2(\Fq)]$ is either $q-1$ or $\tfrac{q-1}2$.
  In the former case, the conjugacy class of $\sigma$ in $GL_2(\Fq)$ is a single conjugacy class in $SL_2(\Fq)$.
  In the latter case, the conjugacy class of $\sigma$ in $GL_2(\Fq)$ is a union of two conjugacy classes in $SL_2(\Fq)$, represented by $\sigma$ and $\mat \epsilon 001 \sigma \mat \epsilon 001\inv$ respectively.
\end{theorem}
\begin{xca}
  Prove Theorem~\ref{theorem:orbit-splitting}.
\end{xca}
\begin{xca}
  Show that the conjugacy classes in $SL_2(\Fq)$ are as follows:
  \begin{enumerate}
  \item $2$ \emph{central} classes, represented by $\mat{\pm 1}00{\pm 1}$.
  \item $4$ \emph{non-semisimple} classes, represented by $\mat{\pm 1}10{\pm1}$ and $\mat{\pm 1}\epsilon 0{\pm 1}$.
  \item $\half(q-3)$ \emph{split regular semisimple} classes, represented by $\mat a00{a\inv}$, $a\in \Fq^*$.
  \item $\half(q-1)2$ \emph{anisotropic semisimple} classes, represented by $\mat 0{-1}1a$, where $\lambda^2-a\lambda+1$ is an irreducible polynomial in $\Fq[t]$.
  \end{enumerate}
\end{xca}
In all, there are
\begin{equation}
  \label{eq:sl2-classes}
  2+4+\frac{q-3}2+\frac{q-1}2
\end{equation}
conjugacy classes.
\begin{table}[htb]
  \caption{Conjugacy classes of $SL_2(\Fq)$}
  \label{table:SL2-conjugacy-classes}
  \begin{tabular}{|p{2.2cm}|c|c|c|c|}
    \hline
    Name & representative & centraliser & \cbox{1.5cm}{no. of classes} & size of class \\
    \hline \hline
    Central & $\mat {\pm 1}00{\pm 1}$ & $q(q-1)(q+1)$ & $2$ & $1$\\
    \hline
    Non- & $\mat {\pm 1}10{\pm 1}$  & $2q$ & $2$ & $\frac{(q-1)(q+1)}{2}$\\
    \cline{2-5}
    semisimple & $\mat {\pm 1}{\epsilon}0{\pm 1}$ & $2q$ & $2$ & $\frac{(q-1)(q+1)}{2}$ \\
    \hline
    Split regular semisimple & \cbox{2.4cm}{\begin{center}$\mat a00{a\inv}$\\ $a\in \Fq^*\setminus \{\pm 1\}$\end{center}} & $q-1$ & $\frac{q-3}{2}$ & $q(q+1)$ \\
    \hline
    Anisotropic regular semisimple & \cbox{2.4cm}{$C_p$, $p[t]\in \Fq[t]$ irreducible, $p(0)=1$} & $q+1$ & $\frac{q-1}{2}$ & $q(q-1)$ \\
    \hline
  \end{tabular}
\end{table}
\section{Parabolically induced representations for $SL_2(\Fq)$}
\label{sec:parab-induc-repr}
Let $B$ now consist of the upper triangular matrices in $SL_2(\Fq)$, $N$ the upper triangular matrices with $1$'s along the diagonal, and $T$ the matrices in $SL_2(\Fq)$ which are diagonal.
Note that the results of Exercise~\ref{xca:subgr-upper-triang} are still valid, as is the Bruhat decomposition:
\begin{equation*}
  SL_2(\Fq)=B\cup BwB, \text{ a disjoint union.}
\end{equation*}
Given a character $\chi$ of $\Fq^*$, we may think of it as a character of $T$ by
\begin{equation*}
  \chi\mat y00{y\inv} = \chi(y).
\end{equation*}
It can be extended to a character of $B$ which is trivial on $N$ by setting
\begin{equation*}
  \chi \mat yx0{y\inv} = \chi(y).
\end{equation*}
Let $I(\chi)$ be the representation of $SL_2(\Fq)$ induced from this character of $B$.
There is an analogue of Proposition~\ref{prop:intertwiners} for $SL_2(\Fq)$.
\begin{prop}
  \label{prop:induced_representations_SL}
  Let $\chi$ and $\mu$ be characters of $\Fq^*$.
  Then,
  \begin{equation*}
    \dim \Hom_{SL_2(\Fq)}(I(\chi),I(\mu))=e_1+e_w,
  \end{equation*}
  where,
  \begin{equation*}
    e_1 =
    \begin{cases} 
      1 & \text{ if } \mu = \chi,\\
      0 & \text{ otherwise,}
    \end{cases}
    \quad
    \text{and}
    \quad
        e_w =
    \begin{cases}
      1 & \text{ if } \chi=\mu\inv,\\
      0 & \text{ otherwise.}
    \end{cases}
  \end{equation*}
\end{prop}
Taking $\mu=\chi$ in Proposition~\ref{prop:induced_representations_SL} gives that
\begin{equation*}
  \dim\End_{SL_2(\Fq)}(I(\chi))=
  \begin{cases} 2 & \text{ if } \chi = \chi\inv\\
    1 & \text{ otherwise.}
  \end{cases}
\end{equation*}
Take $\epsilon\in \Fq^*$ to be a generator (this is cyclic of even order $q-1$ by Theorem~\ref{theorem:multiplicative_group}).
Note that $\chi$ is completely determined by $\chi(\epsilon)$, which can be any $(q-1)$st root of unity in $\C^*$.
Furthermore, $\chi=\chi\inv$ if and only if $\chi(\epsilon)=\chi(\epsilon)\inv$, i.e., if and only if $\chi(\epsilon)=\pm1$.
Therefore, there are $q-3$ characters $\chi$ for which $I(\chi)$ is irreducible.
For each of these, $I(\chi)\cong I(\chi\inv)$, and there are no other isomorphic pairs.
We get $\tfrac{q-3}2$ such irreducible representations, each of degree $q+1$.
There remain the characters $\chi$ for which $\chi(\epsilon)=\pm1$.
Each of these give rise to two irreducible non-isomorphic representations.
We consider the two cases separately:
\subsection*{Case $\chi(\epsilon)=1$}
In this case, $I(\chi)$ contains the invariant one dimensional subspace of constant functions on $G$.
Therefore $I(\chi)$ splits into a direct sum of two irreducible representations, the trivial representation and a representation of dimension $q$, which is called the \emph{Steinberg representation}\index{myindex}{Steinberg representation}.
\subsection*{Case $\chi(\epsilon)=-1$}
In this case it is necessary to make a closer analysis of $\End_{SL_2(\Fq)}(I(\chi))$.
Let $\Delta_1$ denote the unique function $SL_2(\Fq)\to \C$ for which $\Delta_1(1)=1$, $\Delta_1(w)=0$, and $\Delta(b_1gb_2)=\chi(b_1)\Delta(g)\chi(b_2)$.
Also let $\Delta_w$ denote the unique function $SL_2(\Fq)\to \C$ for which $\Delta_w(1)=0$, $\Delta_w(w)=1$, and $\Delta(b_1gb_2)=\chi(b_1)\Delta(g)\chi(b_2)$.
These two functions form a basis of $\End_{SL_2(\Fq)}(I(\chi))$.
Write $I(\chi)=\rho^+\oplus \rho^-$, where $\rho^+$ and $\rho^-$ are the two irreducible summands of of $I(\chi)$.
The identity endomorphism in $I(\chi)$ can be written as a sum of two idempotents, coming from the identity endomorphisms of $\rho^+$ and $\rho^-$.
\begin{xca}
  Show that the identity endomorphism of $I(\chi)$ is given by $f\mapsto q\inv (q-1)\inv \Delta_1*f$.
\end{xca}
\begin{xca}
  Show that
  \begin{align*}
    \Delta_1*\Delta_1=q(q-1)\Delta_1, \quad &\Delta_1*\Delta_w=q(q-1)\Delta_w,\\
    \Delta_w*\Delta_1=q(q-1)\Delta_w, \quad &\Delta_w*\Delta_w=q^2(q-1)\chi(-1)\Delta_1.
  \end{align*}
\end{xca}
\begin{xca}
  Besides $q\inv(q-1)\inv \Delta_1$ and $0$, show that the only idempotents in $\End_{SL_2(\Fq)}(I(\chi))$ are
  \begin{equation*}
    \half q\inv (q-1)\inv(\Delta_1\pm (\sqrt{-1})^\kappa q\inv \Delta_w).
  \end{equation*}
  Here $\kappa=0$ if $\chi(-1)=1$ and $\kappa=1$ if $\chi(-1)=-1$.
\end{xca}
Let ${}^\epsilon I(\chi)$ be the representation of $SL_2(\Fq)$ on the representation space of $I(\chi)$, but where the action of $SL_2(\Fq)$ is given by 
\begin{equation*}
  \mat abcd f(x)=f\left(x \mat \epsilon 001\inv \mat abcd \mat\epsilon 001\right).
\end{equation*}
\begin{xca}
  Show that $f\mapsto \tilde f$, where $\tilde f(x)=f\left(x\mat \epsilon 001\right)$ is an isomorphism $I(\chi)\to {}^\epsilon I(\chi)$.
\end{xca}
\begin{xca}
  Show that $\tilde\Delta_1=\Delta_1$ and $\tilde \Delta_w=-\Delta_w$.
  Conclude that ${}^\epsilon \rho^+=\rho^-$ and ${}^\epsilon \rho^-=\rho^+$.
\end{xca}
Therefore, the two representations $\rho^+$ and $\rho^-$ must have equal degrees.
It follows that $I(\chi)$ is a sum of two irreducible representations, each of degree $\tfrac{q+1}2$.

In this section, we have constructed
\begin{equation*}
  2+2+\frac{q-3}2
\end{equation*}
irreducible representations of $SL_2(\Fq)$.
Comparing with (\ref{eq:sl2-classes}) we see that there remain $2+\frac{q-1}2$ irreducible representations to construct.
The sums of squares of the degrees of the representations that we have constructed so far is:
\begin{equation*}
  \frac{q-3}{2}(q+1)^2+2\left(\frac{q+2}{2}\right)^2+1+q^2.
\end{equation*}
The order of $SL_2(\Fq)$ is $q^3-q$.
The sums of the degrees of the irreducible representations that remain is therefore,
\begin{equation*}
  2\left(\frac{q-1}2\right)^2+(q-1)^2\frac{q-1}2.
\end{equation*}
We will see in Section~\ref{sec:cusp-repr-sl_2fq} that, among the representations that remain to be constructed, there are two of degree $\frac{q-1}2$, and $\frac{q-1}2$ of degree $q-1$.
\chapter{Construction of the cuspidal representations}
\section{Projective Representations and Central Extensions}
\label{sec:extensions}
Let $G$ be a finite group and let $\Hilbert$ be a Hilbert space.
Denote by $U(\Hilbert)$ the group of unitary automorphisms of $\Hilbert$.
Let $U(1)$ denote the group $\{z\in \C\;:\: |z|=1\}$ under multiplication.
\begin{defn}
  [Projective representation]
  \index{myindex}{representation!projective}
  A \emph{projective representation} of $G$ on $\Hilbert$ is a function $\eta:G\to U(\Hilbert)$ such that for every $g,h\in G$, there exists a constant $c(g,h)\in U(1)$ such that
  \begin{equation}
    \label{eq:cocycle-definition}
    \index{myindex}{cocycle!of a projective representation}
    \eta(gh)=c(g,h)\eta(g)\eta(h).
  \end{equation}
\end{defn}
A projective representation where $c(g,h)=1$ is a representation in the sense of Section~\ref{sec:defintions} and, for emphasis, will be called an \lq\lq ordinary representation\rq\rq.
\begin{xca}
  Use the associative law on $G$ to show that the function $c:G\times G\to U(1)$ defined above satisfies the \emph{cocycle condition}:
  \begin{equation}
    \label{eq:cocycle-condition}
    \index{myindex}{cocycle!condition}
    c(g,h)c(gh,k)=c(g,hk)c(h,k).
  \end{equation}
\end{xca}
It is natural to ask whether, given a projective representation $\eta$, is it possible to find suitable scalars $s(g)\in U(1)$ for each $g\in G$ such that $\eta(g)s(g)$ is an ordinary representation.
If such a set of scalars did exist, it would mean that
\begin{equation*}
  \eta(g)s(g)\eta(h)s(h)=\eta(gh)s(gh)
\end{equation*}
for all $g,h\in G$.
Applying (\ref{eq:cocycle-definition}) gives the \emph{coboundary condition}\index{myindex}{coboundary condition}:
\begin{equation*}
  \label{eq:coboundary-condition}
  s(g)s(h)=c(g,h)s(gh).
\end{equation*}
This motivates the following definitions:
\begin{defn}\hspace{0cm}
  \begin{enumerate}
  \item The abelian group of $2$-cocycles of $G$ in $U(1)$ consists of functions $c:G\times G\to U(1)$ which satisfy (\ref{eq:cocycle-condition}).
    This group is denoted $Z^2(G,U(1))$.
  \item Given a function $s:G\to U(1)$, its \emph{coboundary} is defined as the cocycle $c(g,h)=s(g)\inv s(h)\inv s(gh)$.
    The subgroup of $Z^2(G,U(1))$ consisting of all coboundaries is denoted $B^2(G,U(1))$.
  \item The second cohomology group of $G$ with coefficients in $U(1)$ is the quotient $H^2(G,U(1))=Z^2(G,U(1))/B^2(G,U(1))$.
  \end{enumerate}
\end{defn}
Observe that
\begin{prop}
  For any projective representation $\eta$ of $G$, there exists a function $s:G\to U(1)$ such that $\eta(g)s(g)$ is an ordinary representation if and only if the the cocycle defined by (\ref{eq:cocycle-definition}) is a coboundary.
\end{prop}
\begin{defn}
  [Central Extension]
  \index{myindex}{central extension}
  A \emph{central extension} of $G$ by $U(1)$ is a group $\tilde G$, together with a short exact sequence
  \begin{equation*}
    1\to U(1)\to \tilde G \to G \to 1
  \end{equation*}
  such that $U(1)$ is contained in the centre of $\tilde G$.
\end{defn}
Given a central extension $\tilde G$ of $G$ by $U(1)$, pick any function $s:G\to \tilde G$ (which may not be a homomorphism) such that the image of $s(g)$ in $G$ is again $g$. 
Such a function is called a \emph{section}.
The failure of $s$ to be a homomorphism is measured by
\begin{equation}
  \label{eq:cocycle-extension}
  \index{myindex}{cocycle!of a central extension}
  c(g,h)=s(gh)s(h)\inv s(g)\inv \in U(1).
\end{equation}
\begin{xca}
  Show that $c(g,h)$ defined in (\ref{eq:cocycle-extension}) satisfies the cocycle condition (\ref{eq:cocycle-condition}).
  Moreover, if $s$ is replaced by another section $s'$, and $c'$ is the resulting cocycle, then $c'c\inv$ is a coboundary.
\end{xca}
Thus a central extension of $G$ by $U(1)$ determines a well-defined element of $H^2(G,U(1))$.
\begin{xca}
  \label{xca:central-extension}
  Given a cocycle $c:G\times G\to U(1)$ satisfying (\ref{eq:cocycle-condition}), show that $G(c)=G\times U(1)$ with multiplication defined by
  \begin{equation*}
    (g,z)(g',z')=(gg',zz'c(g,g')\inv),
  \end{equation*}
  is a central extension of $G$ by $U(1)$.
  Moreover if $c':G\times G\to U(1)$ is another cocycle, there is an isomorphism $\alpha:G(c)\to G(c')$ such that the diagram
  \begin{equation*}
    \xymatrix{
      1 \ar[r] & U(1) \ar[r] \ar@{=}[d] & G(c) \ar[r] \ar[d]_{\alpha} & G \ar[r] \ar@{=}[d] & 1 \\
      1 \ar[r] & U(1) \ar[r] & G(c') \ar[r] & G \ar[r] & 1
    }
  \end{equation*}
  commutes if and only if $c'c\inv$ is a coboundary.
\end{xca}
In this way, $H^2(G,U(1))$ classifies the central extensions of $G$ by $U(1)$.
Thus, $H^2(G,U(1))$ arises in two different contexts:
\begin{enumerate}
\item It measures the obstruction to modifying a projective representation to an ordinary representation.
\item It classifies the central extensions of $G$ by $U(1)$.
\end{enumerate}
The two are related in the following way:
\begin{xca}
  \label{xca:extensions-and-projective-representations}
  If $\eta$ is a projective representation and $c$ is the cocycle associated to it by (\ref{eq:cocycle-definition}), then $\tilde \eta:G(c)\to U(\Hilbert)$ defined by $\tilde \eta(g,z)=z\eta(g)$ defines an ordinary representation of $G(c)$. 
\end{xca}
In other words, every projective representation can be resolved into an ordinary representation of the central extension corresponding to its cocycle.
\section{The Heisenberg group}
Assume that the finite group $G$ is abelian.
Let $L^2(G)$ denote the Hilbert space obtained when the space of complex valued functions on $G$ is endowed with the Hermitian inner product $\sum_x f(x)\ol{g(x)}$.
On $L^2(G)$, there are two natural families of unitary operators:
\begin{tabbing}
  Translation operators: \hspace{1cm} \= $(T_xf)(y)=f(y-x)$,\hspace{1cm} \= $x\in G$,\\
  Modulation operators: \> $(M_\chi f)(y)=\chi(y)f(y)$, \> $\chi\in \hat G$.
\end{tabbing}
The translation operators give a unitary representation of $G$ on the Hilbert space $L^2(G)$.
The modulation operators give a unitary representation of $\hat G$ on the same space.
However, these operators do not commute:
\begin{xca}
  \label{xca:commutator}
  Show that
  \begin{equation*}
    [T_x,M_\chi]f=\chi(-x)f \text{ for each } f\in L^2(G).
  \end{equation*}
\end{xca}
The commutator is a scalar.
Thus the map $\eta:G\times \hat G\to U(L^2(G))$ defined by
\begin{equation*}
  \eta(x,\chi) = T_x M_\chi
\end{equation*}
defines a projective representation of $G\times \hat G$ on $L^2(G)$.
\begin{xca}
  Show that the cocycle of $G\times \hat G$ with coefficients in $U(1)$ associated to $\eta$ in (\ref{eq:cocycle-definition}) is given by
  \begin{equation}\label{eq:Heisenberg-cocycle}
    c((x,\chi),(x',\chi'))=\chi(x')\inv.
  \end{equation}
\end{xca}
\begin{defn}
  [Heisenberg group]
  \index{myindex}{Heisenberg group}
  The Heisenberg group $H(G)$ of $G$ is the central extension of $G\times \hat G$ by $U(1)$ corresponding to the cocycle (\ref{eq:Heisenberg-cocycle}) (see Exercise~\ref{xca:central-extension}).
\end{defn}
Explicitly, $H(G)$ is the group whose underlying set of points is $G\times \hat G\times U(1)$ with multiplication given by
\begin{equation}
  \label{eq:Heisenberg-multiplication}
  (x,\chi,z)(x'\chi',z')=(x+x',\chi+\chi',zz'\chi(x')).
\end{equation}
The projective representation $\eta$ of $G\times \hat G$ gives rise to an ordinary representation $\tilde \eta$ of $H(G)$ on $L^2(G)$, known as the \emph{Heisenberg representation}\index{myindex}{Heisenberg representation} (see Exercise~\ref{xca:extensions-and-projective-representations}).
Explicitly, the Heisenberg representation is realized as
\begin{equation}
  \label{eq:Heisenberg-action}
  \tilde\eta(x',\chi',z')f(x)=z'\chi'(x-x')f(x-x').
\end{equation}
\begin{remark}
  In the construction, and in all arguments relating to the Heisenberg group $H(G)$, where $G$ is a finite abelian group, $U(1)$ can be replaced by an appropriate finite subgroup.
  Therefore, we may pretend that $H(G)$ is a finite group.
\end{remark}
\begin{xca}
  \label{xca:normal-subgroup-of-Heisenberg-group}
  Verify that $N:=\{0\}\times \hat G\times U(1)$ and $\hat N:=G\times \{0\} \times U(1)$ are normal subgroups of $H(G)$.
  $Z:=\{0\}\times \{0\}\times U(1)$ is the centre of $H(G)$.
  Here $0$ denotes the identity element of either $G$ or $\hat G$.
\end{xca}
Let $\theta:N\to \C^*$ be the character given by $\theta(0,\chi,z)=z$.
Then the induced representation $\theta^{H(G)}$ is a representation of $H(G)$ on the space
\begin{equation}
\label{eq:I}
  I:=\{f:H(G)\to \C \;|\: f(ng)=\theta(n)f(g) \text{ for all } n\in N, g\in H(G)\}.
\end{equation}
The action of $H(G)$ on $I$ is given by $g'f(g)=g(gg')$.
For each $f\in I$, define $\tilde f(x)=f(-x,0,1)$.
Since the elements $(-x,0,1)$, with $x\in G$ form a complete set of representatives of the cosets in $N\bsl H(G)$, $f\mapsto \tilde f$ is an isomorphism of $I$ onto $L^2(G)$.
Let $g'=(x',\chi',z')$ be an element of $H(G)$
\begin{eqnarray*}
  \widetilde{g'f}(x) & = & g'f(-x,0,1)\\
  & = & f(x'-x,\chi',z')\\
  & = & f((0,\chi',z'\chi'(x'-x)\inv)(x'-x,0,1)\\
  & = & z'\chi'(x'-x)\inv f(x'-x,0,1)\\
  & = & z'\chi'(x-x')\tilde f(x-x').
\end{eqnarray*}
Comparing with (\ref{eq:Heisenberg-action}) shows that $\theta^{H(G)}$ is isomorphic the Heisenberg representation $\tilde \eta$.

Let $\hat \theta : \hat N\to \C^*$ be the character given by $\hat \theta (x,0,z)=z$.
Then $\hat \theta^{H(G)}$ is a representation of $H(G)$ on the space
\begin{equation*}
  \hat I :=\{f:H(G)\to \C \;|\: f(\hat n g)=\hat \theta (\hat n)f(g) \}.
\end{equation*}
For each $f\in \hat I$, define $\tilde f (\chi)=f(0,-\chi,1)$.
Since the elements $(0,-\chi,1)$, with $\chi\in \hat G$ form a complete set of representatives of the cosets in $\hat N\bsl H(G)$, $f\mapsto \tilde f$ defines an isomorphism of $\hat I$ onto $L^2(\hat G)$.
\begin{xca}
  Show that in this realization of $\hat \theta^{H(G)}$ on $L^2(\hat G)$, the action of $H(G)$ is given by
  \begin{equation*}
    ((x',\chi',z')f)(\chi)=z'\chi(x')\inv f(\chi-\chi')
  \end{equation*}
\end{xca}
\begin{xca}
  Show that the Fourier transform\index{myindex}{Fourier transform} $\FT:L^2(G)\to L^2(\hat G)$ defined by
\begin{equation*}
  \FT f (\chi) = \sum_{x\in g} f(x)\ol{\chi(x)},\text{ for }\chi\in \hat G
\end{equation*}
is an isomorphism of $H(G)$\dash representations.
\end{xca}
\begin{theorem}
  \label{theorem:Stone-von-Neumann}
  The representation $\tilde \eta$ is irreducible.
  Every irreducible representation of $H(G)$ on which $Z$ acts by the identity character of $U(1)$ is isomorphic to $\tilde\eta$.
\end{theorem}
\begin{proof}
  The irreducibility of $\tilde \eta$ follows from the following exercise:
  \begin{xca}
    Use Corollary~\ref{cor:Mackeycor} to show that $\theta^{H(G)}$ is irreducible.
  \end{xca}
  Suppose that $\rho$ is an irreducible representation of $H(G)$ on which $Z$ acts by the identity character of $U(1)$.
  By Proposition~\ref{prop:Clifford}, 
  \begin{equation*}
    V_\rho = \bigoplus_{\chi\in \hat N(\rho)} V_\chi,
  \end{equation*}
  where $\hat N(\rho)$ consists of a single $H(G)$\dash orbit of characters of $N$.
  By hypothesis, the restriction of all these characters to $Z$ is the identity character of $U(1)$.
  \begin{xca}
    Show that $H(G)$ acts transitively on the set of characters of $N_1$ whose restriction to $Z$ is the identity character of $U(1)$.
  \end{xca}
  \begin{xca}
    Show that $H(G)_\theta = N$.
  \end{xca}
  Therefore, $\theta \in \hat N(\rho)$, and by Proposition~\ref{prop:Mackey-imprimitivity}, $\rho \cong \theta^{H(G)}$.
\end{proof}

Given an automorphism $\sigma$ of $H(G)$, let ${}^\sigma \tilde \eta$ denote the representation of $H(G)$ on the representation space $V_\eta$ of $\eta$ given by ${}^\sigma \tilde \eta(g)= \tilde \eta({}^{\sigma\inv} g)$.
If $\sigma$ fixes every element of $Z$, then ${}^\sigma\tilde \eta$ is also an irreducible representation of $H(G)$ on which $Z$ acts by the identity character of $U(1)$.
By Theorem~\ref{theorem:Stone-von-Neumann}, $\tilde \eta$ and ${}^\sigma \tilde \eta$ are equivalent.
Therefore, there exists $\nu(\sigma):V_\eta\to V_\eta$ such that 
\begin{equation}
  \label{eq:projint}
  \nu(\sigma) \circ \tilde \eta(g) = {}^\sigma \tilde \eta (g) \circ \nu(\sigma) \text{ for every } g\in H(G).
\end{equation}
Moreover, by Schur's lemma, $\nu(\sigma)$ is uniquely determined modulo a scalar.
Let $B_0(G)$ denote the group of all automorphisms of $H(G)$ which fix the elements of $Z$.
\begin{xca}
  Show that 
  \begin{equation*}
    \nu(\sigma)\circ \nu(\sigma') \circ \tilde \eta(g) = {}^{\sigma'\sigma}\tilde \eta(g)\circ \nu(\sigma)\circ\nu(\sigma').
  \end{equation*}
  Conclude that $\nu(\sigma'\sigma)$ and $\nu(\sigma)\circ \nu(\sigma')$ agree up to multiplication by a scalar.
\end{xca}
It follows that the map $\sigma\mapsto \rho(\sigma)=\nu(\sigma\inv)$ is a projective representation of $B_0(G)$ on $L^2(G)$.
Projective representations of subgroups of $B_0(G)$ constructed in this way are known as \emph{Weil representations}\index{myindex}{Weil representation}.
In order to construct $\nu(\sigma)$ it is helpful to think of the realization of $\tilde \eta$ as $\theta^{H(G)}$.
The underlying vector space is the subspace $I$ (see (\ref{eq:I})) of $\C[H(G)]$.
Let $r$ denote the representation of $H(G)$ on $\C[H(g)]$, where $H(G)$ acts by
\begin{equation*}
  r(g')f(g)=f(gg').
\end{equation*}
It is easy to come up with an isomorphism between $r$ and ${}^\sigma r$, namely
$(\nu_r(\sigma)f)(g)=f({}^\sigma g)$.
Unfortunately, $\nu_r(\sigma)f$ may no longer lie in $I$.
This is rectified by modifying $\nu_r(\sigma)$ by an averaging operation to get $\nu(\sigma)$, as is seen in the following exercise:
\begin{xca}
  \label{xca:intertwiner}
  If $f\in I$, show that the function $\nu(\sigma)f$ defined by
  \begin{equation}
    \label{eq:intertwiner}
    (\nu(\sigma)f)(g)=\sum_{\chi\in \hat G} f({}^\sigma((0,\chi,1)g))
  \end{equation}
  is also in $I$.
  The solution will use the fact that $\sigma$ fixes every element of $Z$.
  Show that $\nu(\sigma)$ defined above satisfies (\ref{eq:projint}).
\end{xca}
\begin{xca}
  \label{xca:auto}
  Let $Q:G\times \hat G \to U(1)$ denote the map 
  \begin{equation*}
    Q((x,\chi),(x',\chi'))=\chi(x').
  \end{equation*}
  Let $\sigma$ be any automorphism of $G\times \hat G$ such that
  \begin{equation*}
    Q(\sigma(x,\chi),\sigma(x',\chi'))=Q((x,\chi),(x',\chi')).
  \end{equation*}
  Then the function $\tilde \sigma: H(G)\to H(G)$ defined by
  \begin{equation*}
    \tilde \sigma (x,\chi,z)=(\sigma(x,\chi),z)
  \end{equation*}
  is an automorphism of $H(G)$.
\end{xca}
\begin{xca}
  [Symplectic form of the Heisenberg group]
  \index{myindex}{Heisenberg group!symplectic form}
  \label{xca:symplectic-form}
  Assume that $x\mapsto 2x$ is an automorphism of $G$.
  Consider the bijection $\phi:H(G)\to G\times \hat G\times U(1)$ given by
  \begin{equation*}
    \phi(x,\chi,z)=(x,\chi,z\chi(-\tfrac x2)).
  \end{equation*}
  The multiplication map $m:H(G)^2\to H(G)$ gives rise to a new multiplication map $m':(G\times \hat G \times U(1))^2\to G\times \hat G \times U(1)$ determined by the commutativity of the diagram
  \begin{equation*}
    \xymatrix{
      H(G)^2\ar[r]^m \ar[d]_{\phi\times \phi} & H(G) \ar[d]^{\phi} \\
      (G\times \hat G\times U(1))^2 \ar[r]^{m'} & G\times \hat G \times U(1)
    }
  \end{equation*}
  Show that 
  \begin{equation}
    \label{eq:Heisenberg-multiplication-symplectic-form}
    m'((x,\chi,z),(x',\chi',z'))=(x+x',\chi+\chi',zz'\chi(\tfrac{x'}2)\chi'(-\tfrac x2)).
  \end{equation}
\end{xca}
\section{A special Weil representation}
\label{sec:weil-representation}
In this section $SL_2(\Fq)$ will be realized as a subgroup of $B_0(G)$ for $G=\FQ$.
The resulting Weil representation will turn out to be an ordinary representation (Proposition~\ref{prop:true}).
All the cuspidal representations of $GL_2(\Fq)$ and $SL_2(\Fq)$ will be found inside this representation in Sections~\ref{sec:constr-cusp-repr} and \ref{sec:cusp-repr-sl_2fq} respectively.
Let $G$ be the additive group of $\FQ$.
The map $x\mapsto (y\mapsto \psi(\tr(\ol xy)))$ defines an isomorphism of $\FQ$ onto $\hat{\FQ}$ by Proposition~\ref{prop:additive-Pontryagin-dual}.
Using this identification, the Heisenberg group $H(\FQ)$ can be realized as $\FQ\times \FQ\times U(1)$, with multiplication
\begin{equation*}
  m((x,y,z),(x',y',z'))=(x+x',y+y',zz'\psi(\tr(\ol y x'))).
\end{equation*}
In the symplectic form (see Exercise~\ref{xca:symplectic-form}), multiplication is given by
\begin{equation*}
  m'((x,y,z),(x',y',z'))=(x+x',y+y',zz'\psi(\tr(\half(\ol y x'-\ol{y'}x)))).
\end{equation*}
To go from the Heisenberg group to its symplectic form, the transformation is given by $\phi(x,y,z)=(x,y,z\psi(\tr(\half \ol y x)))$.
Suppose $\sigma=\mat abcd\in SL_2(\Fq)$.
Then if $Q((x,y),(x',y'))=\tr(\half (\ol y x' - \ol{y'}x))$, 
\begin{equation*}
  Q((ax+by,cx+dy),(ax'+by',cx'+dy'))=Q((x,y),(x',y')).
\end{equation*}
It follows that 
\begin{equation*}
  (x,y,z)\mapsto (ax+by,cx+dy,z)
\end{equation*}
defines an automorphism of the symplectic form of the Heisenberg group.
Using $\phi$, we may associate to $\sigma$ the automorphism
\begin{equation}
  \label{eq:action}
  (x,y,z)\mapsto (ax+by,cx+dy,z\psi(\half\tr(-\ol y x + (\ol{cx+dy})(ax+by)))).
\end{equation}
of the Heisenberg group $H(G)$ in its usual coordinates.
\begin{xca}
Show that in the action defined by (\ref{eq:action}), $t(a)=\mat a00{a\inv}$, when $a\in \Fq^*$, acts by
\begin{equation*}
  (x,y,z)\mapsto (ax,a\inv y,z),
\end{equation*}
$w=\mat 01{-1}0$ acts by
\begin{equation*}
  (x,y,z)\mapsto (y,-x,z\psi(\half \tr(-\ol xy-\ol yx))),
\end{equation*}
and $u(c)=\mat 10c1$, when $c\in \Fq$, acts by
\begin{equation*}
  (x,y,z)\mapsto (x,cx+y,z\psi(\half \tr(c\ol x x))).
\end{equation*}
\end{xca}
In the present context, (\ref{eq:intertwiner}) gives 
\begin{equation*}
  (\nu(\sigma)f)(-x,0,1) = \frac{1}{q^2} \sum_{y\in \FQ} f({}^\sigma(-x,y,\psi(\tr(-\ol y x)))).
\end{equation*}
Now,
\begin{multline*}
  {}^\sigma(-x,y,\psi(\tr(-\ol y x))) \\ =  (-ax+by,-cx+dy,\psi(\tr(-\ol y x + \half ( \ol y x + (\ol{-cx+dy})(-ax+by)))))\\
  = (0,-cx+dy,\psi(\half \tr(-\ol y x - (\ol{-cx+dy})(-ax+by))))(-ax+by,0,1).
\end{multline*}
Therefore,
\begin{equation*}
  f({}^\sigma(-x,y,\psi(\tr(-\ol y x)))) = \psi(\half \tr(-\ol yx - (\ol {-cx+dy})(-ax+by)))f(-ax+by,0,1).
\end{equation*}
Therefore,
\begin{multline*}
  (\nu(\sigma)f)(-x,0,1) \\= \frac{1}{q^2} \sum_{y\in \FQ} \psi(\half \tr(-\ol yx - (\ol {-cx+dy})(-ax+by)))f(-ax+by,0,1).
\end{multline*}
Therefore, in the realization of $\tilde\eta$ as $L^2(\FQ)$,
\begin{equation*}
  (\nu(\sigma)\tilde f)(x)=\frac{1}{q^2}\sum_{y\in \FQ} \psi(\half\tr(-\ol yx-(\ol{-cx+dy})(-ax+by)))\tilde f(ax-by)
\end{equation*}
for each $\tilde f\in L^2(G)$, and therefore,
\begin{equation}
  \label{eq:special_intertwiner}
  (\rho(\sigma)\tilde f)(x)=\frac{1}{q^2}\sum_{y\in \FQ} \psi(\half\tr(-\ol yx-(\ol{cx+ay})(-dx-by)))\tilde f(dx+by).
\end{equation}
\begin{xca}
  Show that, for any $\tilde f\in L^2(G)$,
  \begin{equation*}
    (\rho(\sigma)\tilde f)(x)=
    \begin{cases}
      \psi(dcN(x)) \tilde f(dx) & \text{ if } b=0,\\
      \frac{1}{q^2}\sum_{y\in \FQ} \psi(\tfrac{dN(x)-\tr(\ol y x)+aN(y)}{b})\tilde f(y) & \text{ otherwise.}
    \end{cases}
  \end{equation*}
\end{xca}
We have already seen that $\rho:SL_2(\Fq)\to GL(L^2(\FQ))$ is a projective representation.
Let $\tilde \rho$ be the modification of $\rho$ by scalars given by
\begin{equation}
  \label{eq:Weil}
  \tilde \rho (\sigma) \tilde f(x)=
    \begin{cases}
      \psi(dcN(x)) \tilde f(dx) & \text{ if } b=0,\\
      -\frac{1}{q}\sum_{y\in \FQ} \psi(\tfrac{dN(x)-\tr(\ol y x)+aN(y)}{b})\tilde f(y) & \text{ otherwise.}
    \end{cases}
\end{equation}
\begin{prop}
  \label{prop:true}
  The function $\tilde \rho: SL_2(\Fq)\to GL(L^2(\Fq))$ defined by (\ref{eq:Weil}) is an ordinary representation.
\end{prop}
\begin{proof}
  Suppose $\sigma=\mat abcd$, $\sigma'=\mat{a'}{b'}{c'}{d'}$, and $\sigma''=\mat{a''}{b''}{c''}{d''}$ are elements of $SL_2(\Fq)$ such that $\sigma''=\sigma\sigma'$.
  Let $1_0\in L^2(\FQ)$ denote the indicator function of $\{0\}$.
  In the case that $b$, $b'$ and $b''$ are all non-zero, we have
  \begin{equation*}
    \tilde \rho(\sigma'')1_0(0) = -\frac 1q.
  \end{equation*}
  On the other hand,
  \begin{equation*}
    (\tilde \rho(\sigma')1_0)(x)=-\tfrac 1q\psi(d'{b'}\inv N(x)).
  \end{equation*}
  Therefore,
  \begin{eqnarray}
    \nonumber
    (\tilde \rho (\sigma) \tilde \rho (\sigma') 1_0)(0) & = & -\frac 1q \sum_{y\in \FQ} \psi(\tr (ab\inv N(y)))(-\tfrac 1q \psi(d'{b'}\inv N(y)))\\
    \label{eq:expr}
    & = & \frac 1{q^2} \sum_{y\in \FQ} \psi(\tr (ab\inv+d'{b'}\inv) N(y)))\\
    \nonumber
    & = & \frac 1{q^2}\bigg(1 + \sum_{y\in \FQ^*} \psi(\tr (b''b\inv{b'}\inv N(y))\bigg).
  \end{eqnarray}
  Now the norm map $N:\FQ^*\to \Fq^*$ is surjective, and takes each value $q+1$ times (Exercise~\ref{xca:norm}).
  Therefore, as $y$ ranges over $\FQ^*$, $b''b\inv{b'}\inv N(y)$ ranges over $\Fq^*$ $(q+1)$ times. We get
  \begin{eqnarray*}
    \sum_{x\in \FQ^*} \psi(\tr (b''b\inv{b'}\inv N(x))) & = & (q+1)\sum_{u\in \Fq^*} \psi(\tr(u))\\
    & = & (q+1)\sum_{u\in \Fq} \psi(\tr(u)) - (q+1)\\
    & = & -(q+1).
  \end{eqnarray*}
  Therefore,
  \begin{eqnarray*}
    (\tilde \rho(\sigma) \tilde \rho(\sigma') 1_0)(0) & = & \frac{1}{q^2}(1-(q+1))\\
    & = & -\frac 1q.
  \end{eqnarray*}
  We already know that $\tilde \rho(\sigma'')$ and $\tilde \rho(\sigma)\tilde \rho(\sigma')$ differ by a scalar multiple.
  It follows from the above calculations that this scalar multiple is $1$.

  If $b$ and $b'$ are non-zero, but $b''=0$, then $d'{b'}\inv+ab\inv =0$, and the expression (\ref{eq:expr}) equals $1$, which is also the value of $\tilde\rho(\sigma'')1_0(0)$.
  Again, it follows that $\tilde\rho(\sigma'')=\tilde\rho(\sigma)\tilde\rho(\sigma')$.

  When exactly one of $b$ and $b'$ is $0$, then $b''\neq 0$.
  In these cases, $\tilde\rho(\sigma)\tilde\rho(\sigma')=\tilde\rho(\sigma\sigma')=-\tfrac 1q$.
\end{proof}
\begin{xca}
  For $a\in \Fq^*$, let $t(a)=\mat a00{a\inv}$, let $w=\mat 01{-1}0$ and for $c\in \Fq$, let $u(c)=\mat 10c1$.
  Use (\ref{eq:Weil}) to show that for every $\tilde f\in L^2(\FQ)$, 
  \begin{gather}
    \label{eq:Wt}
    (\tilde\rho(t(a))\tilde f)(x) = \tilde f(a\inv x),\\
    \label{eq:Ww}
    (\tilde\rho(w)\tilde f)(x)=\tfrac{-1}q\FT \tilde f(x),\\
    \label{eq:Wu}
    (\tilde\rho(u(c))\tilde f)(x) = \psi(cN(x))\tilde f(x).
  \end{gather}
  Here, the Fourier transform of $\tilde f\in L^2(\FQ)$ is once again thought of as a function of $\FQ$, since $\FQ$ is identified with its Pontryagin dual.
  Explicitly,
  \begin{equation*}
    \FT \tilde f(x)=\sum_{y\in \FQ} \tilde f(y)\psi(\tr(-\ol y x)).
  \end{equation*}
\end{xca}
\begin{xca}
  \label{xca:gens}
  Any element of $SL_2(\Fq)$ can be written as a product of elements of the above types.
  Consider the matrix $\mat abcd \in SL_2(\Fq)$.
  If $b=0$, then $d=a\inv$ and $\mat a0cd=t(a)u(ac)$.
  On the other hand, if $b\neq 0$, then
  $\mat abcd = u(d/b)wu(ab)t(b\inv)$.
\end{xca}
\section{The degrees of cuspidal representations}
\label{sec:degrees_cuspidal_representation}
In Chapter~\ref{chap:induced_representations} we constructed all the representations $(\pi,V)$ of $GL_2(\Fq)$ for which 
\begin{equation*}
  \Hom_{GL_2(\Fq)}(\pi,I(\chi_1,\chi_2))\neq 0 \mbox{ for some characters } \chi_1, \chi_2\in \hat{\Fq^*}.
\end{equation*}
Thus for the representations that remain,
\begin{equation}
  \label{eq:cuspidal}
  \Hom_{GL_2(\Fq)}(\pi,I(\chi_1,\chi_2))=0 \mbox{ for all characters } \chi_1, \chi_2\in \hat{\Fq^*}.
\end{equation}
Representations $(\pi,V)$ satisfying \eqref{eq:cuspidal} are known as the \emph{cuspidal} representations of $GL_2(\Fq)$.
By Frobenius reciprocity (Section~\ref{sec:induced}), we have
\begin{equation*}
  \Hom_B(\pi_B,\chi)=0 \mbox{ for all characters } \chi:B\to \C^* \mbox{ such that } \chi_{|N}\equiv 1.
\end{equation*}

Given a representation $(\pi,V)$  of any group $G$, let $V^*$ be the dual space $\Hom_\C(V,\C)$ of $V$.
Let $\pi^*$ be the representation of $G$ on $V^*$ given by
\begin{equation*}
  (\pi^*(g)\xi)(\mathbf{v})=\xi(\pi(g\inv)\mathbf{v}).
\end{equation*}
The representation $(\pi^*,V^*)$ is called the \emph{contragredient} of $(\pi,V)$.

\begin{prop}
  \label{prop:Jacquet}
  A representation $(\pi,V)$ of $GL_2(\Fq)$ is cuspidal if and only if there exists no non-zero vector $\xi\in V^*$ such that
  \begin{equation}
    \label{eq:invariance}
    \pi^*(n)\xi=\xi \mbox{ for all } n\in N.
  \end{equation}
\end{prop}
\begin{proof}
  Suppose $(\pi,V)$ is not cuspidal.
  Then there exists a non-zero element $\xi\in \Hom_B(V,\chi)$ for some $\chi:B\to \C^*$ such that $\chi_{|N}\equiv 1$.
  Such a $\xi$ can be regarded as an element of $V^*$.
  We have, for any $n\in N$ and $\mathbf{v}\in V$,
  \begin{eqnarray*}
    (\pi^*(n)\xi)(\mathbf{v}) & = & \xi(\pi(n\inv)\mathbf{v})\\
    & = & \xi(\chi(n)\mathbf{v})\\
    & = & \xi(\mathbf{v}),
  \end{eqnarray*}
  so that $\xi$ satisfies \eqref{eq:invariance}.
  
  Conversely, look at the space $V^{*N}$ of all vectors in $V^*$ satisfying \eqref{eq:invariance}.
  This space is preserved under the action of $T$ (since $tNt\inv=N$ for all $t\in T$).
  Therefore, one can write
  \begin{equation*}
    V^{*N}=\bigoplus_{\chi\in \hat{T}} V^{*N}_\chi,
  \end{equation*}
  where $V^{*N}_\chi$ is the space of vectors $\mathbf{v}\in V^{*N}$ which transform under $T$ by $\chi$.
  If $V^{*N}\neq 0$, then there exists $\chi$ such that $V^{*N}_\chi\neq 0$.
  Therefore, $\Hom_B(V,\chi)\neq 0$, from which it follows that $(\pi,V)$ is not cuspidal.
\end{proof}
\begin{xca}
  Show that $(\pi,V)$ is cuspidal if and only if $(\pi^*,V^*)$ is cuspidal.
\end{xca}
\begin{cor}
  \label{cor:cuspidal_degree}
  The degree of every cuspidal representation of $GL_2(\Fq)$ is always a multiple of $(q-1)$.
\end{cor}
\begin{proof}
  Suppose that $(\pi,V)$ is a cuspidal representation.
  For each $a\in \Fq$, let $V^*_a$ be the space of all $\xi\in V^*$ such that
  \begin{equation*}
    \pi^*\mat{1}{x}{0}{1}\xi = \psi(ax)\xi.
  \end{equation*}
  Then the map
  \begin{equation*}
    \xi \mapsto \pi^*\mat{t}{0}{0}{1}\xi
  \end{equation*}
  is an isomorphism of $V^*(a)$ with $V^*(ta)$ for all $t\in \Fq^*$.
  Hence for $a\neq 0$, the $q-1$ spaces $V^*(ta)$, with $t\in \Fq^*$ have the same dimension.
  The space $V^*(0)$ is just $V^*(N)$, hence is trivial.
  Therefore the dimension of $V^*$, hence the degree of $V$ must be a multiple of $q-1$.
\end{proof}
From Corollary \ref{cor:cuspidal_degree} and the discussion at the end of Section \ref{sec:induced_representations} it follows that besides the representations constructed in that section, there are exactly $\half(q^2-q)$ irreducible cuspidal representations, each of degree $q-1$.
These representations are constructed in Section~\ref{sec:constr-cusp-repr}.

A cuspidal representation of $SL_2(\Fq)$ can be defined in a similar manner.
A representation $(\pi,V)$ of $SL_2(\Fq)$ is said to be cuspidal if
\begin{equation*}
  \Hom_{SL_2(\Fq)}(\pi,I(\chi))=0\text{ for all characters } \chi \in \in \Fq^*.
\end{equation*}
\begin{xca}
  Verify that Proposition~\ref{prop:Jacquet} continues to hold when $GL_2(\Fq)$ is replaced by $SL_2(\Fq)$.
\end{xca}
However, Corollary~\ref{cor:cuspidal_degree} does not hold as stated
\begin{xca}
  \label{xca:cuspidal-degree-sl2}
  Show that the degree of a cuspidal representation of $SL_2(\Fq)$ is always a multiple of $\frac{q-1}2$.
\end{xca}

\section{Construction of cuspidal representations of $GL_2(\Fq)$}
\label{sec:constr-cusp-repr}
Let $\omega$ be a character of $\FQ^*$ such that $\omega\neq \chi\circ N$ for any character $\chi$ of $\Fq^*$ (here $N$ denotes the norm map $\FQ\to \Fq$).
Such a character is called \emph{primitive}\index{myindex}{primitive character}.
\begin{xca}
  Show that there are $q^2-q$ such characters.
\end{xca}
Let
\begin{equation*}
  (\FQ^*)_1=\{y\in \FQ^* \;|\: N(y)=1\}.
\end{equation*}
\begin{xca}
  \label{xca:primitive}
  Show that a character $\omega:\FQ^*\to \C^*$ is primitive if and only if its restriction to $(\FQ^*)_1$ is non-trivial.
\end{xca}
Define
\begin{equation*}
  W_\omega=\{\tilde f\in L^2(\FQ) \;|\: \tilde f(yx)=\omega(y)\inv \tilde f(x) \mbox{ for all } y\in (\FQ^*)_1\}.
\end{equation*}
\begin{xca}
  Show that $W_\omega$ is preserved by the action of $\tilde\rho(\sigma)$ for every $\sigma\in SL_2(\Fq)$. [Hint: note that if $N(x)=1$, then $\ol x=x\inv$.]
\end{xca}
Therefore, $\tilde\rho$ gives a representation $(\pi_\omega, W_\omega)$ for each such $\omega$.
For any $x\in \FQ$, the set of elements $x'$ such that $N(x')=N(x)$ coincides with the set of elements of the form $x''x$, where $x''\in (\FQ^*)_1$.
Hence, if $f\in W_\omega$, then the value of $\tilde f$ at $x$ determines the value of $\tilde f$ at any element $x'$ with $N(x')=N(x)$.
However, if $x=0$, there is an additional constraint, namely that $\tilde f(0)=\omega(y)\inv \tilde f(0)$ for every $y\in (\FQ^*)_1$.
By Exercise~\ref{xca:primitive}, if $\omega$ is primitive, then it is forced that $\tilde f(0)=0$.
Since there are $q-1$ non-zero values for the norm, we have
\begin{lemma}
  \label{lemma:dimension}
  When $\omega$ is primitive $W_\omega$ has dimension $q-1$.
  For every $\tilde f\in W_\omega$, $\tilde f(0)=0$.
\end{lemma}
Each matrix $\sigma$ in $GL_2(\Fq)$ can be written in a unique way as a product of $\mat{1}{0}{0}{\det(\sigma)}$ and a matrix in $SL_2(\Fq)$.
Define
\begin{equation}
  \label{eq:extension}
  \left(\tilde\rho\mat{1}{0}{0}{a}\tilde f\right)(x)=\omega(\tilde a)\tilde f(\tilde a x),
\end{equation}
where $\tilde a\in \FQ^*$ is chosen so that $N(\tilde a)=a$.
\begin{xca}
  Check that the right hand side of \eqref{eq:extension} does not depend on the choice of $\tilde a$ such that $N(\tilde a)=a$, and that it preserves $W_\omega$ for each primitive $\omega$.
\end{xca}
Extend $\pi_\omega$ to $GL_2(\Fq)$ by $\tilde\rho\left(\mat 100a \sigma\right)=\tilde\rho\mat 100a \tilde\rho(\sigma)$.
For this extended function to be a homomorphism of groups, it is necessary that, for all $a,a'\in \Fq^*$ and all $\sigma,\sigma'\in SL_2(\Fq)$,
\begin{equation}
  \label{eq:homomorphism}
  \tilde\rho\left(\mat{1}{0}{0}{a} \sigma \mat{1}{0}{0}{a'} \sigma'\right)=\tilde\rho\left(\mat{1}{0}{0}{a}\sigma\right)\tilde\rho\left(\mat{1}{0}{0}{a'}\sigma'\right).
\end{equation}
But 
\begin{equation*}\mat{1}{0}{0}{a}\sigma\mat{1}{0}{0}{a'}\sigma'=\mat{1}{0}{0}{aa'}\left[\mat{1}{0}{0}{{a'}\inv}\sigma\mat{1}{0}{0}{a'}\sigma'\right],
\end{equation*}
and $\mat{1}{0}{0}{{a'}\inv}\sigma\mat{1}{0}{0}{a'}\sigma'\in SL_2(\Fq)$.
\begin{xca}
Using this to expand both sides of \eqref{eq:homomorphism} in terms of \eqref{eq:extension}, show that it is sufficient to check that for each $a\in \Fq^*$, $f\in L^2(\FQ)$ and each element $\sigma$ of $SL_2(\Fq)$,
\begin{equation}
  \label{eq:straightforward}
  \tilde\rho\mat{1}{0}{0}{a}\tilde\rho(\sigma)\tilde \rho\mat{1}{0}{0}{a}\inv = \tilde\rho\left(\mat{1}{0}{0}{a}\sigma\mat{1}{0}{0}{a}\inv\right).
\end{equation}
\end{xca}
\begin{xca}
  Verify (\ref{eq:straightforward}) for $\sigma$ of the form $t(a)$, $w$ and $u(c)$ (see Exercise~\ref{xca:gens}).
  Conclude that it holds for all $\sigma\in SL_2(\Fq)$.
\end{xca}
We will denote again by $(\pi_\omega,W_\omega)$ the restriction of $\tilde \rho$ to the subspace $W_\omega$.
\begin{prop}
  For every primitive character $\omega$, the representation $(\pi_\omega,W_\omega)$ is cuspidal.
\end{prop}
\begin{proof}
  We will show that $W_\omega$ contains no non-zero vectors fixed by $\ol N$, the subgroup consisting of matrices of the form $\mat 10c1$, $c\in \Fq$.
  This suffices, for $\tilde f$ is fixed by $N$ if and only if $\pi_\omega(w)\tilde f$ is fixed by $\ol N$.
  Suppose that $\tilde f_0$ is a vector fixed by $\ol N$.
  By Lemma~\ref{lemma:dimension}, $\tilde f_0(0)=0$.
  On the other hand, if $x\in \FQ^*$, then choose $c\in \Fq$ so that $\psi(cN(x))\neq 1$.
  Then, by (\ref{eq:Wu})
  \begin{equation*}
    \tilde f_0(x)=\left(\tilde\rho\mat{1}{0}{c}{1}\tilde f_0\right)(x)=\psi(cN(x))\tilde f_0(x),
  \end{equation*}
  we have $\tilde f_0(x)=0$.
\end{proof}
Clearly, any sub-representation of a cuspidal representation is also cuspidal.
Therefore, by Corollary~\ref{cor:cuspidal_degree} $(\pi_\omega,W_\omega)$ is simple for each $\omega$ of the type considered above.
\begin{lemma}
  \label{lemma:distinctness}
  Let $\omega$ and $\eta$ be two characters of $\FQ^*$ as above.
  If the representations $(\pi_\omega,W_\omega)$ and $(\pi_\eta,W_\eta)$ are isomorphic, then either $\omega=\eta$ or $\omega=\eta\circ F$, where $F$ is the Frobenius automorphism $\FQ^*\to \FQ^*$ (see Section~\ref{sec:Frobenius_norms_traces}).
\end{lemma}
\begin{proof}
  For each $u\in \Fq^*$, fix an element $\tilde u\in \FQ$ such that $N(\tilde u)=u$.
  Let $1_u\in W_\omega$ be the unique function such that $1_u(\tilde u)=2$ and $1_u(x)=0$ if $N(x)\neq u$.
  The set $\{1_u\;|\:u\in \Fq^*\}$ is a basis of $W_\omega$.
  Therefore, for any $\sigma\in GL_2(\Fq)$, $\tr(\pi_\omega(\sigma))=\sum_{u\in \Fq^*}(\pi_\omega(\sigma)1_u)(\tilde u)$.

  For any $a\in \FQ^*$, $\mat a01a = \mat a01{a\inv} \mat 100{a^2}$.
  From (\ref{eq:Weil}) and (\ref{eq:extension}), we have that
  \begin{equation*}
    (\pi_\omega\mat a01a 1_u)(\tilde u) = \omega(a)\psi(a\inv u)1_u(\tilde u).
  \end{equation*}
  Therefore,
  \begin{eqnarray*}
    \tr(\pi_\omega\mat a01a) & = & \sum_{u\in \Fq^*} \omega(a) \psi(a\inv u)\\
    & = & \omega(a)\sum_{u\in \Fq^*}\psi(u)\\
    & = & -\omega(a).
  \end{eqnarray*}
  \begin{xca}
    \label{ex:distinctness}
    Show that if $\omega$ and $\eta$ are two characters of $\FQ^*$, then their restrictions to $\Fq^*$ are equal if and only if either $\omega=\eta$ or $\omega=\eta\circ F$.
  \end{xca}
  If $(\pi_\omega,W_\omega)$ and $(\pi_\eta,W_\eta)$ were isomorphic, then we would have
  \begin{equation*}
    \tr\left(\pi_\omega\mat{a}{1}{0}{a}\right)=\tr\left(\pi_\eta\mat{a}{1}{0}{a}\right),
  \end{equation*}
  which by Exercise~\ref{ex:distinctness} would mean that either $\omega=\eta$ or $\omega=\eta\circ F$.
\end{proof}
\section{The cuspidal representations of $SL_2(\Fq)$}
\label{sec:cusp-repr-sl_2fq}
Let $\omega$ be a non-trivial character if $(\Fq^*)_1$, the subgroup of $\Fq^*$ consisting of elements of norm one (there are exactly $q$ such characters).
As in section~\ref{sec:constr-cusp-repr} define 
\begin{equation*}
  W_\omega = \{ \tilde f\in L^2(\FQ)\;|\:\tilde f(yx)=\omega(y)\inv \tilde f(x) \text{ for all } x\in \FQ\}.
\end{equation*}
Each such character $\omega$ can be extended to a primitive character of $\FQ^*$, and therefore, the $W_\omega$'s are the same as the spaces defined in Section~\ref{sec:constr-cusp-repr}, and are invariant under the representation $\tilde \rho$  of $SL_2(\Fq)$ on $L^2(\FQ)$.
Each such representation is of dimension $q-1$.
Let $\pi_\omega$ denote the representation of $SL_2(\Fq)$ on $W_\omega$.
These are just the restrictions of the representations of $GL_2(\Fq)$ constructed in Section~\ref{sec:constr-cusp-repr} to $SL_2(\Fq)$.
It follows that they are cuspidal.
However, it no longer follows that these representations are irreducible, as the degree of a cuspidal representation of $SL_2(\Fq)$ is only known to be a multiple of $\frac{q-1}2$ by Exercise~\ref{xca:cuspidal-degree-sl2}.

We shall analyze the representations $\pi_\omega$ through their characters.
We already know that $\tr(\pi_\omega(\mat 10c1))=-1$ from the proof of Lemma~\ref{ex:distinctness}.
\begin{xca}
  Show that $\tr\left(\pi_\omega\mat a00{a\inv}\right)=0$ if $a\neq \pm 1$.
\end{xca}
\begin{lemma}
  \label{lemma:cuspidal-character-value}
  For every character $\omega$ of $(\FQ^*)_1$ and $d\in \Fq$ such that $\lambda^2-d\lambda +1$ is irreducible with roots $z$ and $z\inv$ in $\FQ$,
  \begin{equation*}
    \tr(\pi_\omega(\mat 0{-1}1d))=-\omega(z)-\omega(z\inv),
  \end{equation*}
\end{lemma}
\begin{proof}
  By (\ref{eq:Weil}), we have
  \begin{equation*}
    \tilde\rho\mat 0{-1}1d \tilde f(x)= 
    -\frac 1q\sum_{y\in \FQ}\psi(\tr(\ol y x)-dN(x))\tilde f(y).
  \end{equation*}
  Using the notation of Lemma~\ref{ex:distinctness}, we have
  \begin{equation*}
    \tilde\rho\mat 0{-1}1d 1_u(\tilde u)= 
    -\frac 1q\sum_{y\in \FQ}\psi(\tr(\ol y \tilde u)-du)1_u(y).
  \end{equation*}
  Now, $1_u(y)=0$ unless $y=z\tilde u$ for some $z\in (\FQ^*)_1$.
  We have
  \begin{eqnarray*}
    \tilde\rho\mat 0{-1}1d 1_u(\tilde u) & = & 
    -\frac 1q\sum_{z\in (\FQ^*)_1} \psi(\tr(z\inv \ol{\tilde u}\tilde u)-du)\omega(z)\inv\\
    & = & -\frac 1q\sum_{z\in (\FQ^*)_1} \psi(u(z+z\inv)-du)\omega(z)\inv\\
    & = & -\frac 1q\sum_{z\in (\FQ^*)_1} \psi(u(z+z\inv -d))\omega(z)\inv.
  \end{eqnarray*}
  Therefore,
  \begin{eqnarray*}
    \tr(\pi_\omega\mat 0{-1}1d) & = & 
    -\frac 1q\sum_{u\in \Fq^*} \sum_{z\in (\FQ^*)_1} \psi(u(z+z\inv -d))\omega(z)\inv \\
    & = &
    -\frac 1q\sum_{z\in (\FQ^*)_1} \omega(z)\inv \sum_{u\in \Fq^*} \psi(u(z+z\inv -d)).
  \end{eqnarray*}
  If $d\neq z+z\inv$, then
  \begin{equation*}
    \sum_{u\in \Fq^*} \psi(u(z+z\inv -d))=\sum_{u\in \Fq^*} \psi(u)=-1.
  \end{equation*}
  On the other hand, if $d=z+z\inv$, then
  \begin{equation*}
    \sum_{u\in \Fq^*} \psi(u(z+z\inv-d))=q-1.
  \end{equation*}
  Therefore,
  \begin{eqnarray*}
  \tr(\pi_\omega\mat 0{-1}1d) & = & 
  -\frac 1q \sum_{z+z\inv=d} \omega(z)\inv (q+1)-\frac 1q \sum_{z+z\inv \neq d} \omega(z)\inv\\
  & = & 
  -\frac 1q \bigg[ \sum_{z\in \FQ^*} \omega(z)\inv + \sum_{z+z\inv=d} q\omega(z)\inv\bigg]\\
  & = & -\omega(z)-\omega(z)\inv.
  \end{eqnarray*}
\end{proof}
\begin{xca}
  \label{xca:wierd-char}
  Suppose that $\omega$ is the unique non-trivial character of $(\FQ^*)_1$ taking only the values $\pm 1$.
  Show that $\sum_{\sigma\in SL_2(\Fq)} \tr(\pi_\omega(\sigma))= 2(q^3-q)$.
  Conclude that $\pi_\omega(\sigma)$ is a sum of two non-isomorphic irreducible representations of $SL_2(\Fq)$.
\end{xca}
These representations must be irreducible of degree $\frac{q-1}2$ by Exercise~\ref{xca:cuspidal-degree-sl2}.
  Using the book-keeping at the end of Section~\ref{sec:parab-induc-repr}, we see that there remain $\frac{q-1}2$ irreducible representations of $SL_2(\Fq)$.
\begin{xca}
  Define an equivalence relation on the set of non-trivial characters of $(\FQ^*)_2$ by $\omega\sim \omega'$, where $\omega'=\omega\circ F$.
  Here $F$ is the Frobenius automorphism (Section~\ref{sec:Frobenius_norms_traces}).
  Observe that $\tr(\pi_\omega)=\tr(\pi_{\omega'})$.
  Show that the characters of the representations $\pi_\omega$, where $\omega$ runs over the equivalence classes of non-trivial characters of $(\FQ^*)_1$ are pairwise orthogonal.
\end{xca}
It follows that $\pi_\omega$, $\omega$ non-trivial and different from the character considered in Exercise~\ref{xca:wierd-char} give the remaining $\frac{q-1}2$ irreducible representations of $SL_2(\Fq)$.
\chapter{Some remarks on $GL_n(\Fq)$}
In this chapter we state some results on the representation theory of $GL_n(\Fq)$, without proofs, with the intention of motivating further reading.
The construction of representations of $GL_n(\Fq)$ follows that same principles as in the case of $GL_2(\Fq)$.
Parabolic induction (of which the constructions in Chapter~\ref{chap:induced_representations} are examples) is used to construct a large number of irreducible representations of $GL_n(\Fq)$ from representations of $GL_m(\Fq)$, when $m<n$.
The parameterisation of such representations is, in some sense, related to the representation theory of symmetric groups.
The remaining representations are called cuspidal\index{myindex}{cuspidal representation} and are parameterised by the Galois orbits of primitive characters of $\F_{q^n}^*$. 
The irreducible representations come in families, which reflect the parametrisation of conjugacy classes on $GL_n(\Fq)$.
\section{Parabolic Induction}
The process of parabolic induction is best thought of in terms of a graded associative algebra.
Let $R_n$ denote the free abelian group generated by the set of isomorphism classes of irreducible representations of $GL_n(\Fq)$.
Set $R=\oplus_{n=1}^\infty R_n$.
Let $P_{n,n'}$ denote the subgroup of $GL_{n+n'}(\Fq)$ consisting of matrices with block form
\begin{equation*}
  \begin{pmatrix}
    A_{n\times n} & B\\
    0 & A'_{n'\times n'}
  \end{pmatrix}
  ,
\end{equation*}
where $A$ and $A'$ are in $GL_n(\Fq)$ and $GL_{n'}(\Fq)$ respectively, and $B$ is an arbitrary matrix of the appropriate size.
Given representations $(\pi,V)$ and $(\pi',V')$ of $GL_n(\Fq)$ and $GL_{n'}(\Fq)$ respectively,
let $\pi\tilde \otimes \pi'$ be the representation of $P_{n,n'}$ on $V\otimes V'$ defined by
\begin{equation*}
  \pi \tilde\otimes \pi' \Mat AB0{A'} = \pi(A)\otimes \pi'(A').
\end{equation*}
Define $\pi\circ \pi'$ to be the representation of $(\pi\tilde \otimes \pi')^{GL_{n+n'}(\Fq)}$ of $GL_{n+n'}(\Fq)$ \cite[p.403]{MR0072878}.
This binary operation $R_n\times R_{n'}\to R_{n+n'}$ can be extended linearly to $R$.
Green shows that this is a commutative and associative product on $R$.

\section{Cuspidal representations}
\label{sec:cuspidal-reps-gln}

The cuspidal representations of $GL_n(\Fq)$ are those which are disjoint from all representations of the form $\pi'\circ \pi''$, where $\pi'$ and $\pi''$ are irreducible representations of $GL_{n'}(\Fq)$ and $GL_{n''}(\Fq)$, where $n=n'+n''$ and $n'$ and $n''$ are both positive.

Together with the \lq$\circ$\rq{} operation, cuspidal representations generate all of $R$.

The cuspidal representations of $GL_n(\Fq)$ have a nice parametrisation.
A character $\omega$ of $\F_{q^n}^*$ is called \emph{primitive}\index{myindex}{primitive character} if there does not exists any $d|n$ such that $\omega=N\circ \chi$ for any character $\chi$ of $\F_{q^d}^*$.
Here $N$ denotes the norm map $\F_{q^n}\to \F_{q^d}$ (see Section~\ref{sec:Frobenius_norms_traces}).
The Galois group of $\F_{q^n}$ over $\Fq$ acts on the set of primitive characters of $\F_{q^n}$: $\omega^g(x)=\omega({}^gx)$ for an element $g$ of the Galois group, for each $x\in \F_{q^n}$.
\begin{theorem}
  There is a canonical bijective correspondence between the set of Galois orbits of primitive characters of $\F_{q^n}^*$ and isomorphism classes irreducible cuspidal representations of $GL_n(\Fq)$. 
\end{theorem}
It should be noted that the number such orbits is the same as the number of irreducible monic polynomials of degree $n$ with coefficients in $\Fq$.
These correspond precisely to the conjugacy classes of matrices in $GL_n(\Fq)$ with irreducible characteristic polynomial.
Moreover, this correspondence has a nice manifestation in terms of character values.
\begin{theorem}
  Let $f(t)$ is an irreducible monic polynomial of degree $n$ with coefficients in $\Fq$ with roots $z_1,\ldots, z_n$ in $\F_{q^n}$, and let $\omega$ be a primitive character of $\F_{q^n}^*$.
  Let $\pi_\omega$ denote the irreducible cuspidal representation of $GL_n(\Fq)$ corresponding to the Galois orbit of $\omega$.
  Then
  \begin{equation*}
    \tr(\pi_\omega(C_f))=(-1)^{n-1}\sum_{i=1}^n \omega(z_i).
  \end{equation*}
\end{theorem}

The primary decomposition for matrices (Corollary~\ref{cor:primary-decomposition-matrix-version}) has an analogy for representations of $GL_n(\Fq)$.
Fix an irreducible cuspidal representation $\pi$ of some $GL_n(\Fq)$.
Say that a representation $\rho$ of $GL_m(\Fq)$ is $\pi$\dash primary if it is a subrepresentation of some polynomial expression of $\pi$ in $R$.
If $\rho_1,\ldots, \rho_n$ are irreducible primary representations, with $\rho_i$ begin $\pi_i$\dash primary, where $\pi_1,\ldots,\pi_n$ are pairwise non-isomorphic cuspidal representations, then $\rho_1\circ \cdots \circ \rho_n$ is irreducible.

Green shows that the irreducible $\pi$\dash primary representations are parameterised by partitions.
It is no coincidence that the irreducible representations of symmetric groups are also  parameterised by partitions.
An elegant approach to understanding these relationships is by putting additional structure on $R$, namely that of a \emph{positive self adjoint Hopf algebra}.
Very general results about the structure of such algebras are interpreted in terms of the representation theory of general linear groups over finite fields by Zelevinsky in \cite{MR643482}. 

\appendix
\chapter{Similarity Classes of Matrices}
The classification of representations of $GL_n(\Fq)$ is closely analogous to the classification of conjugacy classes.
The results in this chapter give a classification of the conjugacy classes in $GL_n(\Fq)$, along with representatives for each class.
Descriptions of the centralisers are also given.
\section{Basic properties of matrices}
Let $F$ be any field.
\begin{defn}
  \label{defn:similarity}\index{myindex}{similar matrices}
  Two matrices $A$ and $B$ with entries in $F$ are said to be \emph{similar} if there exists an invertible matrix $X$ such that $BX=XA$.
\end{defn}
Similarity is an equivalence relation on the set of all $n\times n$ matrices.
The equivalence classes are called \emph{similarity classes}\index{myindex}{similarity class}.
Given a matrix $A\in M_n(F)$, for every vector $\vec x\in F^n$ and every polynomial $f(t)\in F[t]$ define $f\vec x=f(A)\vec x$.
This endows $F^n$ with the structure of an $F[t]$\dash module, which will be denoted by $M^A$.
\begin{xca}
  If $A$ is similar to $B$, then $M^A$ is isomorphic to $M^B$ as an $F[t]$\dash module.
\end{xca}
Conversely, given an $F[t]$\dash module $M$, pick any basis of $M$ as an $F$\dash vector space.
Let $A_M$ be the matrix by which $t$ acts on $M$ with respect to this basis.
A different basis of $M$ would give rise to a matrix similar to $A_M$.
Therefore, $M$ determines a similarity class of matrices.
\begin{prop}
  \label{prop:module-similarity}
  $A\mapsto M^A$ gives rise to a bijection between the set of similarity classes of matrices and the set of isomorphism classes of $F[t]$\dash modules.
\end{prop}
\begin{defn}
  [Simple matrix]
  \index{myindex}{simple matrix}
  Recall that an $F[t]$\dash module is called \emph{simple} if there is no non-trivial proper subspace of $M$ which is preserved by $F[t]$.
  A matrix $A$ is said to be \emph{simple} if $M^A$ is a simple $F[t]$\dash module.
\end{defn}
\begin{xca}
  \label{xca:simple}
  Show that $A$ is simple if and only if its characteristic polynomial is irreducible.
\end{xca}
\begin{xca}
  \label{xca:direct-sum}
  For any two matrices $A$ and $B$, let $A\oplus B$ denote the block matrix $\left(
    \begin{smallmatrix}
      A & 0 \\ 0 & B
    \end{smallmatrix}
    \right)$.
    $A\oplus B$ will be called the \emph{direct sum} of $A$ and $B$.
    Show that $M^{A\oplus B}=M^A\oplus M^B$ (a canonical isomorphism of $F[t]$\dash modules).
\end{xca}
\begin{defn}
  [Indecomposable matrix]
  \index{myindex}{indecomposable matrix}
  A matrix is said to be \emph{indecomposable} if it is not similar to a matrix of the form $A\oplus B$, where $A$ and $B$ are two strictly smaller matrices.
  Equivalently, $A$ is indecomposable if $M^A$ is indecomposable as an $F[t]$-module.
\end{defn}
\begin{defn}
  [Semisimple matrix]
  \index{myindex}{semisimple matrix}
  A matrix is said to be \emph{semi\-simple} if it is similar to a direct sum of simple matrices.
  Equivalently, $A$ is semisimple if $M^A$ is a semisimple $F[t]$\dash module (i.e., $M^A$ is a direct sum of simple $F[t]$\dash modules).
\end{defn}
\begin{xca}
  For any $\lambda\in F$, show that the matrix $\mat \lambda 1 0 \lambda$ is indecomposable, but not semisimple (and hence not simple either).
\end{xca}
\section{Primary decomposition}
\label{sec:primary-decomposition}
Let $f(t)$ be any irreducible monic polynomial in $F[t]$.
Given an $F[t]$\dash module $M$, its \emph{$f$\dash primary part}\index{myindex}{primary part} is the submodule
\begin{equation*}
  M_f=\{\vec x\in M\;:\:f^k\vec x=0\text{ for some } k\in \N\}.
\end{equation*}
\begin{theorem}
  [Primary decomposition]
  \cite[Theorem~3.11]{basicalgebra1}
  \label{theorem:primary-decomposition}
  \index{myindex}{primary decomposition theorem}
  Let $M$ be an $F[t]$\dash module which is also a finite dimensional $F$\dash vector space.
  Then $M_f=0$ for all but finitely many irreducible monic polynomials $f(t)\in F[t]$.
  \begin{equation*}
    M=\bigoplus_f M_f,
  \end{equation*}
  the sum being over all the irreducible monic polynomials $f$ for which $M_f\neq 0$.
\end{theorem}
Let $f\in F[t]$ be an irreducible monic polynomial.
An $F[t]$-module $M$ is called \emph{$f$\dash primary}\index{myindex}{primary module} if $M=M_f$.
$M$ is called \emph{primary} if it is $f$\dash primary for some $f$.
\begin{xca}
  \label{xca:primary}
  Let $f(t)\in F[t]$ be an irreducible monic polynomial, and $p(t)\in F[t]$ be any monic polynomial.
  Show that $F[t]/p(t)$ is $f$\dash primary if and only if $p(t)=f(t)^r$ for some $r\geq 0$.
\end{xca}
\begin{theorem}
  \label{theorem:primary-factors}
  Let $f(t)\in F[t]$ be an irreducible monic polynomial, and $A$ be a square matrix.
  Then $M^A_f \neq 0$ if and only if $f(t)$ divides the characteristic polynomial of $A$.
\end{theorem}
\begin{proof}
  Let $\chi_A$ denote the characteristic polynomial of $A$.
  If $f$ is an irreducible polynomial that does not divide $\chi_A$, then there exist polynomials $r$ and $s$ such that $fr+\chi_As=1$.
  Evaluating at $A$ and applying the Cayley-Hamilton theorem shows that $f(A)r(A)=I$.
  It follows that $f(A)$ is non-singular.
  Hence $f(A)^k$ is also non-singular for every positive integer $k$.
  Therefore, $M^A_f=0$.

  Conversely, if $M^A_f=0$, then $f(A)^k$ is non-singular for every $k\in \N$.
  In particular, $f(A)$ is non-singular.
  Let $E$ be a splitting field of $f$.
  Suppose that
  \begin{equation*}
    f(t)=\prod_{i=1}^h (t-\mu_i)^{m_i},
  \end{equation*}
  with $\mu_1,\ldots,\mu_h\in E$ distinct, and $m_1,\ldots,m_h\in \N$.
  Therefore,
  \begin{equation*}
    f(A)=\prod_{i=1}^h (A-\mu_iI)^{m_i}.
  \end{equation*}
  Since $f(A)$ is non-singular, so is $A-\mu_iI$ for each $i$.
  Therefore, no $\mu_i$ is an eigenvalue of $A$. 
  It follows that $f$ does not divide $\chi_A$.
\end{proof}
If $M^A$ is $f$\dash primary then the matrix $A$ is called an \emph{$f$\dash primary matrix}\index{myindex}{primary matrix}.
It follows that a matrix is primary if and only if its characteristic polynomial has a unique irreducible factor.
\begin{cor}
  \label{cor:primary-decomposition-matrix-version}
  Every matrix $A\in M_n(F)$ is similar to a matrix of the form
  \begin{equation*}
    \bigoplus_{f|\chi_A} A_f,
  \end{equation*}
  where $A_f$ is an $f$\dash primary matrix, and the sum is over the irreducible factors of the characteristic polynomial of $A$.
  Moreover, for every $f$, the similarity class of $A_f$ is uniquely determined by the similarity class of $A$.
\end{cor}
Thus, the study of similarity classes of matrices is reduced to the study of similarity classes of primary matrices.
\section{Structure of a primary matrix}
\label{sec:structure-primary}
\begin{theorem}
  [Structure theorem]
  \index{myindex}{Structure theorem}
  \cite[Section~3.8]{basicalgebra1}
  For every $F[t]$-module $M$, there exist non-constant monic polynomials $f_1,\ldots,f_r$ such that $f_1|\cdots |f_r$ and
  \begin{equation*}
    M\cong F[t]/f_1(t)\oplus \cdots \oplus F[t]/f_r(t).
  \end{equation*}
\end{theorem}
Fix an irreducible monic polynomial $f(t)\in F[t]$.
If $M$ is $f$\dash primary, then by Exercise~\ref{xca:primary}, each for each $i$, $f_i=f^{\lambda_i}$ for some $\lambda_i>0$.
Therefore,
\begin{cor}
  [Structure of a primary module]
  \index{myindex}{primary module!structure of}
  \label{cor:primary-structure}
  If $M$ is an $f$\dash primary $F[t]$\dash module, then there exists a non-decreasing sequence of integers $\lambda_1\leq \cdots \leq \lambda_r$ such that
  \begin{equation*}
    M\cong F[t]/f(t)^{\lambda_1}\oplus \cdots \oplus F[t]/f(t)^{\lambda_r}.
  \end{equation*}
\end{cor}
\begin{defn}
  [Partition]
  \index{myindex}{partition}
  A \emph{partition} is a finite sequence $\lambda=(\lambda_1, \cdots, \lambda_r)$ of positive integers such that $\lambda_1\leq \cdots \leq \lambda_r$.
  Define $|\lambda|:=\lambda_1+\cdots +\lambda_r$.
  One says that \emph{$\lambda$ is a partition of $|\lambda|$}.
  The \emph{length}\index{myindex}{partition!length of} of $\lambda$ is the non-negative integer $r$ (there is an \lq empty partition\rq{} of length $0$ denoted $\emptyset$, with $|\emptyset|=0$).
  Let $\Lambda$ denote the set of all partitions.
\end{defn}
Given a partition $\lambda=(\lambda_1,\ldots,\lambda_l)$, define an $F[t]$-module
\begin{equation*}
  M_{f,\lambda}=F[t]/f(t)^{\lambda_1}\oplus\cdots\oplus F[t]/f(t)^{\lambda_l}.
\end{equation*}
Corollary~\ref{cor:primary-structure} says that every $f$\dash primary $F[t]$\dash module is isomorphic to $M_{f,\lambda}$ for some partition $\lambda$.
\begin{xca}
  Suppose that $f$ and $f'$ are two irreducible monic polynomials, $\lambda$ and $\lambda'$ two partitions.
  Show that the $F[t]$-modules $M_{f,\lambda}$ and $M_{f',\lambda'}$ are isomorphic if and only if $f=f'$ and $\lambda=\lambda'$.
\end{xca}
Let $S$ denote the set of all irreducible monic polynomials in $F[t]$.
Given a function $\psi:S\to \Lambda$ such that $\psi(f)=\emptyset$ for all but finitely many $f\in S$, let $M_\phi$ denote the $F[t]$-module
\begin{equation*}
  M_\psi=\bigoplus_{f\in S}M_{f,\psi(f)}.
\end{equation*}
Then $\dim_FM_\psi=\sum_{f\in S} \deg(f)|\psi(f)|$.
Let $n_\psi=\dim_FM_\psi$.
\begin{theorem}
  [Similarity classes of matrices]
  \label{theorem:similarity-classes}
  The map $\psi\mapsto M_\psi$ is a bijective correspondence between the set of all functions $S\to \Lambda$ with the property that $\psi(f)=\emptyset$ for all but finitely many $f\in S$ and $n_\psi=n$ and the set of isomorphism classes of $n$-dimensional $F[t]$\dash modules (and hence the set of similarity classes of $n\times n$ matrices).
\end{theorem}
\section{Block Jordan canonical form}
There is a version of the Jordan canonical form for matrices for which the irreducible factors of the characteristic polynomial have derivatives which are not identically zero.

In order to obtain this form, we need the following result:
\begin{theorem}
  \label{theorem:Henselization}
  Suppose that $f$ an irreducible monic polynomial in $F[t]$ such that $f'(t)$ is not identically zero.
  Let $E$ denote the field $F[t]/f(t)$.
  Then the rings $k[t]/f(t)^r$ and $E[u]/u^r$ are isomorphic.
\end{theorem}
\begin{proof}
  The main step in the proof is a version of \emph{Hensel's Lemma}
  \begin{lemma}
    [Hensel]
    There exists $q_r(t)\in F[t]$ such that $q_r(t)\equiv t \mod f(t)$, and $f(q_r(t))\equiv 0 \mod f(t)^r$.
  \end{lemma}
  \begin{proof}
    The proof is by induction on $r$.
    When $r=1$, one may take $q_1(t)=t$.

    Suppose that $q_{r-1}(t)\in F[t]$ is such that
    \begin{equation*}
      q_{r-1}(t)\equiv t \mod f(t) \quad \text{and}\quad f(q_{r-1}(t))\equiv 0 \mod f(t)^{r-1}.
    \end{equation*}
    Then, using the Taylor expansion, for any $h(t)\in F[t]$, 
    \begin{equation*}
      f(q_{r-1}(t)+f(t)^{r-1}h(t))\equiv f(q_{r-1}(t))+f(t)^{r-1}h(t)f'(q_{r-1}(t)) \mod f(t)^r.
    \end{equation*}
    Since $q_{r-1}(t)\equiv t\mod f(t)$, $f'(q_{r-1}(t)\equiv f'(t) \mod f(t)$.
    By hypothesis $f'(t)$ is a non-zero polynomial of degree strictly less than $f(t)$.
    Therefore, $f'(t)$ is not divisible by $f(t)$.
    Since $f(t)$ is irreducible, it means that there exists $r(t), s(t)\in F[t]$ such that $f'r+fs=1$, which means that $f'(t)r(t) \equiv 1\mod f(t)$.
    Since $f(q_{r-1}(t))\equiv 0 \mod f(t)^{r-1}$, there exists $f_1(t)\in F[t]$ such that
    \begin{equation*}
      f(q_{r-1}(t))=f(t)^{r-1}f_1(t).
    \end{equation*}
    When $h(t)=-f_1(t)r(t)$ and $q_r(t)=q_{r-1}(t)+f(t)^{r-1}h(t)$, one has
    \begin{equation*}
      q_r(t)\equiv t \mod f(t) \quad \text{and}\quad f(q_r(t))\equiv 0 \mod f(t)^r.
    \end{equation*}
  \end{proof}
  Given $q_r(t)$ as in Hensel's lemma, the map
  \begin{equation*}
    \phi: F[u,v]/(f(v),u^r)\to F[t]/f(t)^r
  \end{equation*}
  given by setting $\phi(v)=q_r(t)$, and $\phi(u)=f(t)$ gives rise to a well defined ring homomorphism, since $f(q_r(t))\equiv 0\mod f(t)^r$
  Since $q_r(t)\equiv t\mod f(t)$, and $f(t)$ and $q_r(t)$ lie in the image of $\phi$, $t$ also lies in the image of $\phi$.
  This makes $\phi$ surjective.
  Moreover, $\phi$ is a linear transformation of $F$-vector spaces of dimension $rd$.
  Therefore, $\phi$ must be an isomorphism of rings.
\end{proof}
\begin{defn}
  [Companion matrix]
  \index{myindex}{companion matrix}
  Let $f(t)=t^n-a_{n-1}t^{n-1}-\cdots-a_1t-a_0$.
  Then the \emph{companion matrix} of $f$ is the $n\times n$ matrix:
  \begin{equation*}
    C_f=
    \begin{pmatrix}
      0 & 0 & \cdots & 0 & a_0\\
      1 & 0 & \cdots & 0 & a_1\\
      0 & 1 & \cdots & 0 & a_2\\
      \vdots & \vdots & \ddots & \vdots & \vdots\\
      0 & 0 & \cdots & 1 & a_{n-1}
    \end{pmatrix}
    .
  \end{equation*}
\end{defn}
\begin{theorem}
  [Block Jordan Canonical Form]
  \index{myindex}{Jordan canonical form}
  \label{theorem:JCF}
  Let $A\in M_n(F)$ be such that for every irreducible factor $f$ of the characteristic polynomial of $A$, $f'$ is not identically zero.
  Then $A$ can be written as a block diagonal matrix with blocks of the form
  \begin{equation*}
    J_r(f)=
    \begin{pmatrix} 
    C_f & 0 & 0 & \cdots & 0 & 0\\
    I & C_f & 0 & \cdots & 0 & 0\\
    0 & I & C_f & \cdots & 0 & 0\\
    \vdots & \vdots & \vdots & \ddots & \vdots & \vdots\\
    0 & 0 & 0 & \cdots & C_f & 0\\
    0 & 0 & 0 & \cdots & I & C_f
    \end{pmatrix}
    _{rd\times rd},
  \end{equation*}
  where $d$ is the degree of $f$, an irreducible factor of the characteristic polynomial of $A$, $C_f$ is the companion matrix of $f$, and $r$ is a positive integer.
  Up to rearrangement of blocks, this canonical form is unique. 
\end{theorem}
\begin{proof}
  By Exercise~\ref{xca:direct-sum} and Theorem~\ref{theorem:primary-decomposition} one may assume that $A$ is $f$\dash primary, for some irreducible monic polynomial $f$.
  Let $E=F[v]/f(v)$.
  By Corollary ~\ref{cor:primary-structure} and Theorem~\ref{cor:primary-structure}, there exists a partition $\lambda$ such that
  \begin{equation*}
    M^A\cong E[u]/u^{\lambda_1}\oplus\cdots \oplus E[u]/u^{\lambda_r}.
  \end{equation*}
  In the notation of the proof of Theorem~\ref{theorem:Henselization}, let $\theta(t)=t-q(t)$, where we write $q$ for $q_{\lambda_i}$ for some $i$.
  Then $\theta(t)\in (f(t))$.
  But $\theta(t) \notin (f(t)^2)$, for if it did, we would have
  \begin{eqnarray*}
    f(t) & = & f(\theta(t)+q(t))\\
    & \cong & f(q(t))+\theta(t)f'(q(t) \mod (f(t)^2)\\
    & = & 0 \mod (f(t)^2),
  \end{eqnarray*}
  a contradiction.
  Therefore, $\theta(t)=\alpha f(t)$, where $\alpha$ is a unit in $F[t]/f(t)^{\lambda_i}$.
  In the isomorphism
  \begin{equation*}
    F[t]/(f(t)^r)\to E[u]/u^r=F[u,v]/(u^r,f(v)),
  \end{equation*}
  $t\mapsto \alpha u +v$.
  Since $A$ acts by $t$, 
  with respect to the basis of $E[u]/u^{\lambda_i}$ over $F$ given by
  \begin{equation*}
    1, v,\ldots, v^{d-1},\alpha, \alpha v,\ldots, \alpha v^{d-1},\ldots, \alpha^{\lambda_i-1}, \alpha^{\lambda_i-1}v,\ldots \alpha^{\lambda_i-1}v^{d-1},
  \end{equation*}
  the matrix of multiplication by $t=\alpha u +v$ is $J_{\lambda_i}(f)$.
\end{proof}
The hypothesis on $A$ in Theorem~\ref{theorem:JCF} always holds when $F$ is a perfect field, as we shall see in Section~\ref{sec:perfect-fields}.
By Corollary~\ref{cor:perfection} every finite field is perfect.
Therefore, every matrix over a finite field has a Jordan canonical form.

\section{Centralisers}
\label{sec:centralizers}
For any $A\in M_n(F)$ define
\begin{equation*}
  Z(A)=\{B\in M_n(A)\;|\:AB=BA\}.
\end{equation*}
\begin{theorem}
  Let $A\in M_n(F)$ be a matrix such that for each irreducible factor $f$ of the characteristic polynomial of $A$, $f'$ is not identically zero.
  Suppose that $A$ is similar to $\oplus_f A_f$, where $A_f$ is $f$\dash primary (see Corollary~\ref{cor:primary-decomposition-matrix-version}).
  Then $Z(A)\cong \oplus_f Z(A_f)$.
  If $A$ is $f$\dash primary, $E=F[t]/f(t)$, and $\lambda$ is the partition associated to $M^A$ in Corollary~\ref{cor:primary-structure}, then
  \begin{equation*}
    Z(A)\cong \End_{E[u]}(E[u]/u^{\lambda_1}\oplus\cdots \oplus E[u]/u^{\lambda_r}).
  \end{equation*}
\end{theorem}
Note that the group of units of the centraliser algebra $Z(A)$ will be the centraliser of $A$ in $GL_n(F)$.
\begin{proof}
  The theorem follows easily from Theorem~\ref{theorem:Henselization}, using the fact that $\End_{F[t]}M^A\cong Z(A)$.
\end{proof}
\section{Perfect fields}
\label{sec:perfect-fields}
\begin{defn}
  \label{defn:perfect-field}\index{myindex}{perfect field}
  A \emph{perfect field} is either a field of characteristic zero, or a field of characteristic $p>0$ for which the map $x\mapsto x^p$ is bijective.
\end{defn}
\begin{theorem}
  \label{theorem:irreducible-polynomials-over-perfect-fields}
  Suppose that $F$ is a perfect field and $f(t)\in F[t]$ is a non-constant irreducible polynomial.
  Then $f'(t)$ does not vanish identically.
\end{theorem}
\begin{proof}
  If $f'=0$, then the characteristic of $F$ must be $p>0$ and $f$ must be of the form
  \begin{equation*}
    f(t)=a_0+a_1t^p+a_2t^{2p}+\cdots .
  \end{equation*}
  Since $F$ is perfect, there exist $b_i\in E$ such that $b_i^p=a_i$.
  Then
  \begin{equation*}
    f(t)=(b_0+b_1t+b_2t^2+\cdots)^p,
  \end{equation*}
  contradicting the irreducibility of $f$.
\end{proof}
\chapter{Finite Fields}
\label{chap:finite_fields}
In this section, we study the finite fields.
Such a field must have prime characteristic (the characteristic of a field is the smallest integer $n$ such that $1+1+\cdots 1$ ($n$ times) is $0$).
Therefore, it contains one of the finite fields $\Fp$.
This makes it a finite dimensional vector space over $\Fp$, so that its order must be some power of $p$.
We will see that, up to isomorphism, there is exactly one field of a given prime power order.
We will also show that choosing a non-trivial character of the additive group of a finite field gives an identification of this group with its Pontryagin dual, and we will study the Fourier transform in this context.

\section{Existence and uniqueness}
\label{sec:existence_uniqueness}

We will show that for any power $p^k$ of $p$, there is a unique finite field of order $p^k$, which is unique up to isomorphism\footnote{The method given here assumes the existence of an algebraic closure of $\Fp$. This is contingent upon the axiom of choice. However, there are other ways to prove the same results without using the axiom of choice, see \cite[Chapter~7]{MR1070716}.}.
For convenience, write $q=p^k$.
Fix an algebraic closure $\overline{\Fp}$ of $\Fp$.
Look at the set
\begin{equation*}
  \Fq:=\{x\in \overline{\Fp}|x^q=x\}.
\end{equation*}
\begin{xca}
  If $x,y\in \Fq$, then show that $x+y$ and $xy$ are in $\Fq$.
\end{xca}
It follows from the above exercise that $\Fq$ is a field (why?).
\begin{xca}
  \label{exercise:roots}
  Let $K$ be any field, and $f(X)\in K[X]$ be of degree $d$. Show that $f(X)$ can not have more than $d$ roots in $K$.
\end{xca}
Since the elements of $S$ are roots of the polynomial $X^q-X$ which has degree $q$, there can be no more than $q$ of them.
\begin{xca}
  Let $K$ be any field. For a polynomial $f(X)\in K[X]$
  \begin{equation*}
    f(X)=a_0X^n+a_1X^{n-1}+\cdots+a_{n-1}X+a_n
  \end{equation*}
  define its (formal) derivative to be the polynomial
  \begin{equation*}
    f'(X)=na_0X^{n-1}+(n-1)a_1X^{n-2}+\cdots+a_{n-1}.
  \end{equation*}
  Show that if $a$ is a multiple root of $f(X)$ (i.e., $(X-a)^2|f(X)$) then $f'(a)=0$.
\end{xca}
The derivative of the polynomial $X^q-X$ is the constant polynomial $-1$.
Therefore, all its roots in $\overline{\Fp}$ are distinct.
This means that $S$ has exactly $q$ elements.
Therefore there exists a subfield of order $q$ in $\overline{\Fp}$.
In particular there exists a finite field of order $q$.

On the other hand, in any field of order $q$, the multiplicative group of non-zero elements in the field has order $q-1$.
Therefore, each element of the field satisfies $x^{q-1}=1$, or $x^q=x$.
Thus any subfield of $\Fq$ of order $q$ must be equal to $S$.

Now any field of order $q$ must have characteristic $p$, hence is an algebraic extension of $\Fp$.
Therefore, it is isomorphic to some subfield of $\ol \Fp$.
We have seen that only such field is $\Fq$.
It follows that every field of order $q$ is isomorphic to $\Fq$.
We have proved the following theorem:
\begin{theorem}
  \label{theorem:finite_fields}
  For every power $q$ of a prime number, there exists a finite field of order $q$, which is unique up to isomorphism.
\end{theorem}

\section{The multiplicative group of $\Fq$}
\label{sec:multiplicative_group}
We present the proof of the following theorem straight out of Serre's book \cite{MR0344216}.
\begin{theorem}
  \label{theorem:multiplicative_group}
  The multiplicative group $\Fq^*$ is cyclic of order $q-1$.
\end{theorem}
\begin{proof}
  If $d$ is an integer $\geq 1$, then let $\phi(d)$ denote the number of integers $x$ with $1\leq x\leq d$ such that $(x,d)=1$.
  In other words, the image of $x$ in $\Z/d\Z$ is a generator of $\Z/d\Z$.
  The function $\phi(d)$ is called the \emph{Euler totient function}.
  \begin{lemma}
    \label{lemma:totient}
    If $n\geq 1$ is an integer then
    \begin{equation*}
      n=\sum_{d|n}\phi(d).
    \end{equation*}
  \end{lemma}
  \begin{proof}
    If $d|n$,
    let $C_d$ denote the unique subgroup of order $d$ in $\Z/n\Z$, and $\Phi_d$ denote the generators of $C_d$.
    Then $\Z/n\Z$ is the disjoint union of the $\Phi_d$.
    $\Phi_d$ had $\phi(d)$ elements.
    Adding up cardinalities, $n=\sum_{d|n}\phi(d)$.
  \end{proof}
  \begin{lemma}
    Let $H$ be a finite group of order $n$.
    Suppose that, for all divisors $d$ of $n$ the set
    \begin{equation*}
      \{x\in H | x^d=1\}
    \end{equation*}
    has at most $d$ elements. Then $H$ is cyclic.
  \end{lemma}
  \begin{proof}
    Let $d|n$.
    If there exists $x\in H$ of order $d$, the subgroup
    \begin{equation*}
      \langle x \rangle = \{1,x,\ldots,x^{d-1}\}
    \end{equation*}
    is cyclic of order $d$.
    By hypothesis, every element $y$ such that $y^d=1$ is in $\langle x \rangle$.
    In particular, the elements of order $d$ are the generators of $\langle x \rangle$, and these are $\phi(d)$ in number.
    Hence the number of elements of order $d$ is either $0$ or $\phi(d)$.
    If it were zero for some $d|n$, Lemma~\ref{lemma:totient} would show that the number of elements in $H$ is strictly less than $n$, contrary to hypothesis.
    In particular, there exists an element of order $n$ in $H$, and $H$ so $H$ is cyclic of order $n$.
  \end{proof}
  To complete the proof of Theorem~\ref{theorem:multiplicative_group}, note that the equation $x^d=1$ is a polynomial equation, and hence, by Exercise~\ref{exercise:roots} has at most $d$ solutions in $\Fq$.
\end{proof}
\begin{cor}
  \label{cor:perfection}
  Every finite field is perfect.
\end{cor}
\section{Galois theoretic properties}
\label{sec:Frobenius_norms_traces}
In general, if $E$ is an extension of a field $F$, then every element $x\in E$ can be thought of as an $F$-linear endomorphism of the $F$\dash vector space $E$, when it acts on $E$ by multiplication.
The trace of this map is denoted $\tr_{E/F}(x)$. The function $\tr_{E/F}:E\to F$ is called the \emph{trace function of $E$ over $F$}.
Likewise, the determinant of multiplication by $x$ is denoted $N_{E/F}(x)$. The function $N_{E/F}:E\to F$ is called the \emph{norm map of $E$ over $F$}.

Since $\FQ$ is a quadratic extension of $\Fq$, its Galois group is cyclic of order $2$.
Clearly, the map $F:x\mapsto x^p$ is an automorphism of $\FQ$ that fixes $\Fq$.
Therefore, it must be the non-trivial element in the Galois group of $\FQ$ over $\Fq$.
$F$ is called the \emph{Frobenius automorphism}.
In analogy with complex conjugation, we write $F(x)=\overline{x}$ for each $x\in \FQ$.
\begin{prop}
  Suppose $x\in \FQ$.
  Then $x=\overline{x}$ if and only if $x\in \Fq$.
\end{prop}
Let $N$ and $\tr$ denote the norm and trace maps of $\FQ$ over $\Fq$ respectively.
Then
\begin{equation*}
  N(x)=x\ol x, \quad \tr(x)=x|\ol x.
\end{equation*}
Note that for any $x\in \FQ$, $N(x)=0$ if and only if $x=0$.
\begin{xca}
  \label{xca:norm}
  Show that the norm map $N:\FQ^* \to\Fq^*$ is surjective.
  Conclude that for any $x\in \Fq$, the number of elements $y\in \FQ$ such that $N(y)=x$ is
  \begin{equation*}
    \begin{cases}
      q+1 & \mbox{ if } x\neq 0\\
      1& \mbox{ if } x=0.
    \end{cases}
  \end{equation*}
\end{xca}
\section{Identification with Pontryagin dual}
\label{sec:Pontryagin}
Let $\psi_0:\Fq\to \C^*$ be a non-trivial additive character.
Such a character is completely determined by its value at $1$, which can be any $p$th root of unity different from $1$.
Then $\psi:\Fq\to \C^*$ defined by $\psi(x)=\psi(\tr_{\Fq/\Fp}(x))$ is a non-trivial additive character of $\Fq$.
\begin{prop}
  \label{prop:additive-Pontryagin-dual}
  For each $x'\in \Fq$, set $\psi_{x'}(x)=\psi(x'x)$.
  Then $x'\mapsto \psi_{x'}$ is an isomorphism from the additive group of $\Fq$ onto its Pontryagin dual.
\end{prop}
\begin{proof}
  The map $x'\mapsto \psi_{x'}$ is clearly an injective homomorphism.
  By Proposition~\ref{prop:Pontryagin-duality}, it must also be onto.  
\end{proof}

\backmatter

\bibliographystyle{amsalpha}
\bibliography{refs}

\providecommand{\bysame}{\leavevmode\hbox to3em{\hrulefill}\thinspace}
\providecommand{\MR}{\relax\ifhmode\unskip\space\fi MR }
\providecommand{\MRhref}[2]{%
  \href{http://www.ams.org/mathscinet-getitem?mr=#1}{#2}
}
\providecommand{\href}[2]{#2}
\begin{thebibliography}{Bum97}

\bibitem[Bum97]{MR1431508}
Daniel Bump, \emph{Automorphic forms and representations}, Cambridge Studies in
  Advanced Mathematics, vol.~55, Cambridge University Press, Cambridge, 1997.
  \MR{MR1431508 (97k:11080)}

\bibitem[DL76]{MR0393266}
P.~Deligne and G.~Lusztig, \emph{Representations of reductive groups over
  finite fields}, Ann. of Math. (2) \textbf{103} (1976), no.~1, 103--161.
  \MR{MR0393266 (52 \#14076)}

\bibitem[G{\'e}r75]{MR0396859}
Paul G{\'e}rardin, \emph{Construction de s\'eries discr\`etes {$p$}-adiques},
  Springer-Verlag, Berlin, 1975, Sur les s\'eries discr\`etes non ramifi\'ees
  des groupes r\'eductifs d\'eploy\'es $p$-adiques, Lecture Notes in
  Mathematics, Vol. 462. \MR{MR0396859 (53 \#719)}

\bibitem[Gre55]{MR0072878}
J.~A. Green, \emph{The characters of the finite general linear groups}, Trans.
  Amer. Math. Soc. \textbf{80} (1955), 402--447. \MR{MR0072878 (17,345e)}

\bibitem[IR90]{MR1070716}
Kenneth Ireland and Michael Rosen, \emph{A classical introduction to modern
  number theory}, second ed., Graduate Texts in Mathematics, vol.~84,
  Springer-Verlag, New York, 1990. \MR{MR1070716 (92e:11001)}

\bibitem[Jac84]{basicalgebra1}
Nathan Jacobson, \emph{Basic algebra}, vol.~I, Hindustan Publishing
  Corporation, 1984.

\bibitem[Jor07]{Jordan07}
Herbert~E. Jordan, \emph{Group-characters of various types of linear groups},
  Amer. J. Math. \textbf{29} (1907), 387--405.

\bibitem[Lus84]{MR804741}
George Lusztig, \emph{Characters of reductive groups over finite fields},
  Proceedings of the International Congress of Mathematicians, Vol.\ 1, 2
  (Warsaw, 1983) (Warsaw), PWN, 1984, pp.~877--880. \MR{MR804741 (86i:20062)}

\bibitem[Mac58]{MR0098328}
George~W. Mackey, \emph{Unitary representations of group extensions. {I}}, Acta
  Math. \textbf{99} (1958), 265--311. \MR{MR0098328 (20 \#4789)}

\bibitem[Sch07]{Schur06}
Issai Schur, \emph{Unterschungen \"uber die {D}arstellung der endlichen
  {G}ruppen durch gebrochene lineare {S}ubstitutionen}, J. Reine Angew. Math.
  \textbf{132} (1906-07).

\bibitem[Ser73]{MR0344216}
J.-P. Serre, \emph{A course in arithmetic}, Springer-Verlag, New York, 1973,
  Translated from the French, Graduate Texts in Mathematics, No. 7.
  \MR{MR0344216 (49 \#8956)}

\bibitem[Shi68]{MR0233931}
Takuro Shintani, \emph{On certain square-integrable irreducible unitary
  representations of some {${\mathfrak p}$}-adic linear groups}, J. Math. Soc.
  Japan \textbf{20} (1968), 522--565. \MR{MR0233931 (38 \#2252)}

\bibitem[Spr70]{MR0263942}
T.~A. Springer, \emph{Cusp forms for finite groups}, Seminar on Algebraic
  Groups and Related Finite Groups (The Institute for Advanced Study,
  Princeton, N.J., 1968/69), Lecture Notes in Mathematics, Vol. 131, Springer,
  Berlin, 1970, pp.~97--120. \MR{MR0263942 (41 \#8541)}

\bibitem[Tan67]{MR0219635}
Shun'ichi Tanaka, \emph{Construction and classification of irreducible
  representations of special linear group of the second order over a finite
  field}, Osaka J. Math. \textbf{4} (1967), 65--84. \MR{MR0219635 (36 \#2714)}

\bibitem[Wei64]{MR0165033}
Andr{\'e} Weil, \emph{Sur certains groupes d'op\'erateurs unitaires}, Acta
  Math. \textbf{111} (1964), 143--211. \MR{MR0165033 (29 \#2324)}

\bibitem[Wig39]{MR1503456}
E.~Wigner, \emph{On unitary representations of the inhomogeneous {L}orentz
  group}, Ann. of Math. (2) \textbf{40} (1939), no.~1, 149--204. \MR{MR1503456}

\bibitem[Zel81]{MR643482}
Andrey~V. Zelevinsky, \emph{Representations of finite classical groups},
  Lecture Notes in Mathematics, vol. 869, Springer-Verlag, Berlin, 1981, A Hopf
  algebra approach. \MR{MR643482 (83k:20017)}

\end{thebibliography}

\Printindex{myindex}{Index of terms}

\end{document}